\documentclass[11pt,leqno]{amsart}
\usepackage{amssymb,amsfonts,amsmath,amsopn,amstext,amscd,extarrows,latexsym,xy,hyperref,mathrsfs,verbatim, yhmath}
\usepackage{epsfig}
\input xy
\xyoption{all}

\allowdisplaybreaks[1]

\theoremstyle{remark}
\newtheorem{para}{\bf}[section]

\newtheorem{rmk}[para]{\bf Remark}

\theoremstyle{definition}
\newtheorem{exam}[para]{\bf Example}

\newtheorem{dfn}[para]{\bf Definition}

\theoremstyle{plain}
\newtheorem{thm}[para]{\bf Theorem}
\newtheorem{lemma}[para]{\bf Lemma}
\newtheorem{sublemma}[para]{\bf Sublemma}
\newtheorem{cor}[para]{\bf Corollary}
\newtheorem{prop}[para]{\bf Proposition}

\newenvironment{numequation}
{\addtocounter{enumi}{1}\begin{equation}}{\end{equation}}

\newcommand{\cC}{{\mathcal C}}

\newcommand{\cE}{{\mathcal E}}
\newcommand{\cF}{{\mathcal F}}

\newcommand{\cH}{{\mathcal H}}
\newcommand{\cI}{{\mathcal I}}
\newcommand{\cJ}{{\mathcal J}}

\newcommand{\cM}{{\mathcal M}}

\newcommand{\cO}{{\mathcal O}}

\newcommand{\cW}{{\mathcal W}}

\newcommand{\bB}{{\bf B}}

\newcommand{\bG}{{\bf G}}

\newcommand{\bL}{{\bf L}}
\newcommand{\bM}{{\bf M}}
\newcommand{\bN}{{\bf N}}

\newcommand{\bP}{{\bf P}}

\newcommand{\bT}{{\bf T}}
\newcommand{\bU}{{\bf U}}

\newcommand{\bbC}{{\mathbb C}}

\newcommand{\bbN}{{\mathbb N}}

\newcommand{\bbP}{{\mathbb P}}
\newcommand{\bbQ}{{\mathbb Q}}
\newcommand{\bbR}{{\mathbb R}}

\newcommand{\bbZ}{{\mathbb Z}}

\newcommand{\fra}{{\mathfrak a}}
\newcommand{\frb}{{\mathfrak b}}

\newcommand{\frd}{{\mathfrak d}}

\newcommand{\frg}{{\mathfrak g}}

\newcommand{\frl}{{\mathfrak l}}
\renewcommand{\frm}{{\mathfrak m}}
\newcommand{\frn}{{\mathfrak n}}

\newcommand{\frp}{{\mathfrak p}}
\newcommand{\frq}{{\mathfrak q}}

\newcommand{\frt}{{\mathfrak t}}
\newcommand{\fru}{{\mathfrak u}}

\newcommand{\frx}{{\mathfrak x}}

\newcommand{\GL}{{\rm GL}}
\newcommand{\Hom}{{\rm Hom}}

\newcommand{\Ad}{{\rm Ad}}

\newcommand{\diag}{{\rm diag}}
\newcommand{\hra}{\hookrightarrow}

\newcommand{\im}{{\rm im}}
\newcommand{\ind}{{\rm ind}}
\newcommand{\Ind}{{\rm Ind}}
\newcommand{\Lie}{{\rm Lie}}
\newcommand{\lra}{\longrightarrow}
\newcommand{\midc}{\;|\;}

\newcommand{\Qp}{{\bbQ}_p}
\newcommand{\ra}{\rightarrow}
\newcommand{\Rep}{{\rm Rep}}

\newcommand{\sub}{\subset}
\newcommand{\supp}{{\rm supp}}
\newcommand{\alg}{{\rm alg}}

\newcommand\Zp{{{\bbZ}_p}}

\newcommand{\Pf}{{\it Proof. }}
\renewcommand{\qed}{{\hfill{\space} $\Box$}}

\newcommand{\B}{\mathcal{B}}

\newcommand{\D}{\mathcal{D}}

\renewcommand{\O}{\mathcal{O}}

\newcommand{\h}[1]{\widehat{#1}}

\begin{document}
	
	\title[On some non-principal locally  analytic representations]{On some non-principal locally  analytic representations induced by cuspidal Lie algebra representations}
	\author[Sascha Orlik]{Sascha Orlik, with an appendix by Andreas Bode}
	\address{Fakultät 4 - Mathematik und Naturwissenschaften, Bergische Universit\"at Wuppertal,
		Gau{\ss}stra\ss{}e 20, D-42119 Wuppertal, Germany}
	\email{orlik@uni-wuppertal.de}
	
	\email{abode@uni-wuppertal.de}
	
	\maketitle
	
	\normalsize

	\begin{abstract}
		Let $G$ be a split reductive $p$-adic Lie group. This paper is the first in a series on the construction of locally analytic $G$-representations which do not lie in the principal series.
		Here we consider the case of the general linear group $G=\GL_{n+1}$ and locally analytic representations which are induced by cuspidal modules of the Lie algebra.
		We prove that they are ind-admissible  and  satisfy the homological vanishing criterion in the definition of supercuspidality in the sense of Kohlhaase. 
	\end{abstract}
	
	\tableofcontents
	\section{Introduction}

	
	Let $L$ be a finite extension of $\Qp$ with ring of integers $O_L$. Let $G = \bG(L)$ be the group of $L$-valued points of a split connected reductive algebraic group $\bG$ over $O_L.$ This paper considers some locally analytic $G$-representations which do not lie in the principal series, i.e., which are not closed subrepresentation of parabolically induced representations $\Ind^G_P(W)$ from locally algebraic representations $W$ as in \cite{OS2}.
	We prove  that they are ind-admissible (cf. Definition \ref{dfn_qa}) and satisfy the  homological vanishing criterion in the definition of supercuspidality in the sense of Kohlhaase \cite{K2}. Similarly  to the principal series case \cite{OS2}, the construction proceeds by globalizing certain  representations of the attached Lie algebra. Indeed our representations are induced  by certain cuspidal weight modules which are by the very definition disjoint from the category $\cO$ used in loc.cit..
	
	The motivation for the construction is twofold. The first is intrinsic, since in every representation theory of a given group where a kind of "parabolic" induction exists it is natural to determine the cuspidal representations. 
	Apart from this aspect the second reason is given by  the hypothetical $p$-adic local Langlands correspondence. In fact
	one might wonder whether the locally analytic representations considered here have a meaning, e.g. on the Galois side.

	The theory of locally analytic representations was introduced by P. Schneider and J. Teitelbaum in \cite{ST1}. Such representations appear as the locally analytic vectors in Banach space representations and in the study of vector bundles on  $p$-adic period domains \cite{ST3, O, Sp} inclusive their coverings \cite{DLB,CDN,Van}.   So far, mainly  irreducible parabolically induced locally-analytic representations are considered. To the author's knowledge the aspect of cuspidality 
	was first considered in \cite{KS} by Kisin and Strauch. A few years later Kohlhaase \cite{K2}
	suggested a definition of a supercuspidal representation $V$ by demanding that $V$ is topologically irreducible and that all cohomology groups $H_i(N,V)$ for the  unipotent radicals $N$ of all parabolic subgroups in $G$  vanish. Moreover, he gave an important example of a locally algebraic supercuspidal representation. Those
	appear already in the $p$-adic cohomology of the Drinfeld tower (in the case of $G=\GL_2$) \cite{Van,DLB,CDN} studied by Dospinescu, Le-Bras and Colmez, Dospinescu, Nizio\l{} and Vanhaecke, respectively.

	The idea for the construction is similar to the principal series case. Let $K$ be a finite extension of $L$ which serves as our coefficient field. We start with certain Lie algebra representations of $\frg=\Lie \;G$, i.e., with modules $M$ of the universal enveloping algebra $U(\frg)$ and globalize them to locally analytic representations. More precisely,  we consider for $G=\GL_{n+1}$ and a fixed torus $T$  the category of weight modules $\cW$
	and its subcategory $\cC$  of cuspidal  weight modules $M$ in the sense of Grantcharov and Serganova \cite{GS}.  If these cuspidal modules are simple they coincide with those of Fernando \cite{Fe}. Each cuspidal module has a degree which is simply the common dimension of its non-trivial weight spaces.  For each $\mu \in K^{n+1}$ such that $\mu_i\not\in \bbZ, i=0,\ldots, n,$ we can define explicitly a simple cuspidal module $M^\mu$ of degree $1.$  
	To every finitely generated weight module $M$ we can attach its set of weights in $X^\ast(T)_K\cong K^{n+1}.$ By considering the norms of the entries in $K^{n+1}$ we get a bounded subset $\mu_M \subset K^{n+1}.$ For $\kappa \in \bbR$, we let $\cW_{\leq \kappa}$ be the full subcategory of $\cW$  bounded by $\kappa.$  
	To the category $\cW_\kappa$ we attach a finite field extension $J=J_\kappa$ of $K$ such that all non-trivial weights of $M$ give rise to locally analytic characters $T_0 \to J^\times$ 
	where $T_0\subset T$ is the maximal compact subgroup.
	Let $D(G_0)$ be the distribution algebra of the compact subgroup $G_0:={\bf G}(O_L).$
	In this way we may consider $M$ as a module over the subalgebra $D(\frg,T_0)$  of $D(G_0)$ generated by $\frg$ and $T_0$.
	Further we consider smooth representations $V$ of the maximal compact subgroup $T_0$ of our torus $T$. 
	We set then for $M\in \cW_\kappa$ 
	$$\cF^G_\kappa(M,V):=c-\Ind^G_{G_0} (D(G_0) \otimes_{D(\frg,T_0)} M \otimes V')'$$ where   $c-\Ind^G_{G_0}$ is the compact induction of locally analytic representations. This functor  can be extended via the expression $c-\Ind^G_{G_0} (D(G_0) \otimes_{D(\frg,T_0)} M \otimes Z \otimes V')'$ to consider another ingredient, an algebraic representation $Z$ of the standard parabolic subgroup $P=P_{(1,n)}$ corresponding to the decomposition $(1,n)$ of $n+1.$ Here we refer to chapters 3  and 4 for more details of the construction. Thus we have a functor
	$$\cF^G_\kappa: \cW_\kappa \times \Rep^{alg}_K(P_{(1,n)}) \times  \Rep^\infty_K(T_0) \to \Rep^{la}_J(G)$$
	into the category of locally analytic $G$-representations.
	It turns out that the objects in the essential image of this functor are not admissible in the sense of Schneider and Teitelbaum \cite{ST2}. In fact, they are larger but we will show that they are ind-admissible (see Definition \ref{dfn_qa}, i.e., they are strict inductive limits of admissible representations when restricted to a compact open subgroup).
	
	On of main results of this paper is the following theorem:
	
	\vskip10pt
	
	\noindent {\bf Theorem A.} {\it (i) $\cF^G_\kappa$ is functorial in all arguments: contravariant in $M$ and $Z$, covariant in $V$.
		
		\vskip8pt
		\noindent (ii) $\cF^G_\kappa$ is exact in all arguments.
		
		\vskip8pt
		\noindent  (iii) For all $M\in \cW_\kappa$, for all algebraic representations $Z$ of $P_{(1,n)}$ and for all smooth admissible representations of  $T_0$, the locally analytic representation $\cF^G_\kappa(M,Z,V)$ is ind-admissible and of compact type.
		
		\vskip8pt
		\noindent (iv) Let $(M,Z,V)$ be as in (iii) and suppose additionally that $M$ is  cuspidal. Then for all parabolic subgroups $P\subset G$ with unipotent radical $N$ we have $H_i(N,\cF^G_\kappa(M,Z,V))=0$ for all $i\geq 0.$}  
	
	Concerning topologically irreducibility statements of the representations above, we have to take care on the action of the centre. For this reason we  define the following variant in the case where $M,Z$ and $V$ are irreducible and $M=M^\mu$ is a cuspidal module of degree $1$ corresponding to weight $\mu\in K^{n+1}.$  Then the action of $Z(G)\cap T_0$ on $(M^\mu \otimes Z\otimes_K V')'$ is given by a locally analytic character $\eta_{\mu,Z,V}$. We fix a locally character $\eta$ on $Z(G)$ which extends $\eta_{\mu,Z,V}$ and consider the attached one dimensional dual space $J_\eta'$. Then we set
	$$\overline{X}_\mu(Z,V):=D(Z(G)G_0) \otimes_{D(\frg,T_0)}  (M^\mu \otimes Z\otimes_K V')\otimes J_\eta'.$$
	and
	$$\overline{\cF}_\mu(Z,V):={\rm c-\Ind}^G_{Z(G)G_0}(\overline{X}_\mu(Z,V)').$$
	Then $\overline{X}_\mu(Z,V)$ is a finitely generated  $D(Z(G)G_0)$-module and its topological dual $\overline{X}_\mu(Z,V)'$ is a locally analytic $G_0$-representation. 
	
	The other main theorem  of this paper is:
	
	\noindent {\bf Theorem B.} {\it  Let $n=1$. Suppose that $w(\mu)-\mu \not\in \bbZ^2$ where $w\in W(T,G)$ is the non-trivial reflection in the Weyl group $W(T,G)$. Suppose further that $|\mu_i|\not\in |L|, i=0,1$, $|\mu_0| \neq |\mu_1|$ and that $|\mu_i|\leq 1$ for at least one $i.$    Let $Z$ be given by an algebraic character $\lambda=(\lambda_0,\lambda_1)$ of $T.$ If $\lambda_0+\lambda_1 \neq 0$ we further assume that $|\mu_i|\geq 1$ for at least one $i.$ Let $V$ be irreducible. Then $(\overline{X}_\mu(Z,V)'$ is topologically irreducible. }
	
	\vskip12pt
	
	In particular the representation $\overline{X}_\mu(Z,V)'$ is supercuspidal in the sense of Kohlhaase \cite{K2}. Of course it is tempting to conjecture that some kind of version of Theorem B holds true for general $n$ and where $M$ is a simple cuspidal module and $Z,V$ are irreducible, respectively. Concerning the representation $\overline{\cF}_\mu(Z,V)$ we do not expect to have this property. Indeed, in the appendix it is shown by Bode that there exist non-trivial Hecke operators on $\overline{\cF}_\mu({\bf 1},{\bf 1})$ for $n=2$. We hope that they cut out admissible topologically irreducible $G$-representations.
	
	The proof of the irreducibility result and of Theorem A i), ii) follows the same strategy as in the principal series case \cite{OS2}. As for Theorem B we arrive at a formal power series in one variable and show that it does  not converge by analysing explicitly its coefficients. To generalize this idea to larger $n$ seems to be too complicated to the author at the moment.   As for the proof of iv) we use the fact that the elements in $\frn$ form an Ore set in $U(\frg)$ and in $D(G)$. We prove that for cuspidal $M$ these elements act bijectively on $\cF^G_\kappa(M,Z,V)$, so that we may view $M$ and $\cF^G_\kappa(M,Z,V)$ as modules for the localised algebra $S^{-1}U(\frg)$ where $S=\frn U(\frn)$. Then the statement follows from easy methods of homological algebra.

	As for the list of content we recall in section 2 some basic facts in locally analytic representation theory. In particular we consider the compact induction functor $c-\Ind^G_C$ for a compact open subgroup $C$ of $G.$ Then we define ind-admissible representations and prove some basic properties. Section 3 is devoted to review the necessary background on cuspidal modules over the field of complex numbers. In section 4 we define our functor $\cF^G_\mu$ and prove assertions i)-iii) from the Theorem A above. The irreducibility result is shown in section 5. Finally we prove part iv) in the last section. In the appendix there is a treatment of ind-admissible representations by considering systematically the dual side using the distribution algebra. Finally as mentioned above the reader finds here the construction of some Hecke operators on the above mentioned $G$-representations for $n=2.$  

	{\it Notation and conventions:} We denote by $p$ a prime number and consider fields $L \sub K$ which are both finite extensions of $\Qp$. 
	Let $O_L$ and $O_K$ be the rings of integers of $L$, resp. $K$, and let $|\cdot |_K$ be the absolute value on $K$ such that $|p|_K = p^{-1}$. The field $L$ is our ''base field'', whereas we consider $K$ as our ''coefficient field''. For a locally convex $K$-vector space $V$ we denote by $V'_b$ its strong dual, i.e., the $K$-vector space of continuous linear forms equipped with the strong topology of bounded convergence. Sometimes, in particular when $V$ is finite-dimensional, we simplify notation and write $V'$ instead of $V'_b$. All finite-dimensional $K$-vector spaces are equipped with the unique Hausdorff locally convex topology.
	
	We let $\bG_0$ be a split reductive group scheme over $O_L$ and $\bT_0 \sub \bB_0 \sub \bG_0$ a maximal split torus and a Borel subgroup scheme, respectively. We denote by $\bG$, $\bB$, $\bT$ the base change of $\bG_0$, $\bB_0$ and $\bT_0$ to $L$. By $G_0 = \bG_0(O_L)$, $B_0 = \bB_0(O_L)$, etc., and $G = \bG(L)$, $B = \bB(L)$, etc., we denote the corresponding groups of $O_L$-valued points and $L$-valued points, respectively. Standard parabolic subgroups of $\bG$ (resp. $G$) are those which contain $\bB$ (resp. $B$). For each standard parabolic subgroup $\bP$ (or $P$) we let $\bL_\bP$ (or $L_P$) be the unique Levi subgroup which contains $\bT$ (resp. $T$) and ${\bf U_P}$ its unipotent radical. The opposite unipotent radical is denoted by ${\bf U^-_P}.$  Finally, Gothic letters $\frg$, $\frp$, etc., will denote the Lie algebras of $\bG$, $\bP$, etc.: $\frg = \Lie(\bG)$, $\frt = \Lie(\bT)$, $\frb = \Lie(\bB)$, $\frp = \Lie(\bP)$, $\frl_P = \Lie(\bL_\bP)$, etc.. Base change to $K$ is usually denoted by the subscript ${}_K$, for instance, $\frg_K = \frg \otimes_L K$.
	
	\vskip8pt
	
	We make the general convention that we denote by $U(\frg)$, $U(\frp)$, etc., the corresponding enveloping algebras, {\it after base change to $K$}, i.e., what would be usually denoted by $U(\frg) \otimes_L K$, $U(\frp) \otimes_L K$ etc.
	Similarly, we use the abbreviations $D(G) = D(G,K), D(P) = D(P,K)$ etc. for the locally $L$-analytic distributions with values in $K.$
	
	\vskip10pt

	{\it Acknowledgements.} 
	I am very grateful to Andreas Bode for his careful reading of this paper and for all the discussions on it. In particular for pointing out to me some mistakes in a previous version and the help to remediate them.
	I thank  Georg Linden and Tobias Schmidt for some helpful remarks.
	
	
	\vspace{1cm}
	\section{Preliminaries on locally analytic representations}
	
	\setcounter{enumi}{0}
	We start by recalling some basic facts on locally analytic representations as introduced by Schneider and Teitelbaum \cite{ST1}.
	
	For a locally $L$-analytic group $H$, let $C^{an}(H,K)$ be the locally convex vector space of locally $L$-analytic $K$-valued functions.  The dual space $D(H) = D(H,K) = C^{an}(H,K)'$ is a topological $K$-algebra which has the structure of a Fr\'echet-Stein algebra when $H$ is compact \cite{ST2}.
	More generally, if $V$ is a Hausdorff locally convex $K$-vector space, let $C^{an}(H,V)$ be the $K$-vector space consisting of locally analytic functions with values in $V$. It has the structure of a Hausdorff locally convex vector space, as well.
	
	A locally analytic $H$-representation is a Hausdorff barrelled  locally convex $K$-vector space together with a homomorphism $\rho: H \ra \GL_K(V)$
	such that the action of $H$ on $V$  is continuous and the orbit maps $\rho_v: H \rightarrow V, \; h \mapsto \rho(h)(v)$, are
	elements in $C^{an}(H,V)$ for all $v \in V$. 
	We denote by $\Rep_K^{la}(H)$ the category of locally analytic $H$-representations on $K$-vector spaces where the morphisms are the continuous $H$-linear maps.
	
	We recall that a Hausdorff locally convex $K$-vector space $V$ is called of {\it compact type} if it is an inductive limit of countably many Banach spaces with injective and compact transition maps, cf. \cite[sec. 1]{ST1}. In this case, the strong dual $V'_b$ is a nuclear Fr\'echet space (by \cite[16.10, 19.9]{S1}). We denote by $\Rep_K^{la,c}(H)$ the  full subcategory of $\Rep_K^{la}(H)$ consisting of objects of compact type.  By \cite[3.3]{ST1}, the duality functor gives an equivalence of categories
	
	\begin{equation*}\begin{array}{ccc}\label{equivalence}
			
			\Rep_K^{la,c}(H)
			
			& \stackrel{\sim}{\longrightarrow} &
			
			\left\{\begin{array}{c}
				\mbox{separately continuous $D(H)$-}  \\
				\mbox{modules on nuclear Fr\'echet} \\
				\mbox{spaces with continuous  } \\
				\mbox{$D(H)$-module maps}
			\end{array} \right\}^{op}.
			
		\end{array}
	\end{equation*}
	
	\vskip8pt
	
	\noindent In particular, $V$ is topologically irreducible if $V'_b$ is a topologically simple $D(H,K)$-module. We let $\cC_H$ be the category of coadmissible  $D(H)$-modules, cf. \cite{ST2}. Then by the very definition the above functor induces an equivalence
	$$\Rep_K^{a}(H) \to \cC_H $$
	where $\Rep_K^{a}(H)$ is the subcategory of $\Rep_K^{la,c}(H)$ consisting of admissible $H$-representations.  
	Then   we also recall that a locally analytic $H$-representation $V$ is called {\it strongly admissible} if its strong dual $V'_b$ is a finitely generated $D(H_0)$-module for any (equivalently, one) compact open subgroup $H_0 \sub H$, cf. \cite[sec. 3]{ST1}.
	
	For any closed subgroup $H'$ of $H$ and any locally analytic
	representation $V$ of $H'$, we denote by $\Ind^H_{H'}(V)$ the
	induced locally analytic representation. It is defined by
	$$\Ind^H_{H'}(V) := \Big\{f \in C^{an}(H,V) \midc \forall h' \in H', \forall h \in H: f(h \cdot
	h') = (h')^{-1} \cdot f(h) \; \Big\} \;.$$
	The group $H$ acts on this vector space by $(h \cdot f)(x) = f(h^{-1}x)$. If $V$ is of compact type and $H/H'$ is compact then $\Ind^H_{H'}(V)$ is of compact type again and  $\Ind^H_{H'}(V)' = D(H)\otimes_{D(H')}  V'_b$ is a nuclear Fr\'echet space.
	
	Next we consider the compactly supported induction which is defined similarly as for smooth representations, cf. \cite[Section 2]{ST5}. Suppose that $H$ is second countable and let $C\subset H$ be a compact open subgroup.
	Then we set
	$$c-\Ind^H_C(V):=\{f \in \Ind^H_C(V) \mid f \mbox{ has compact support} \}.$$
	This is an $H$-stable subspace, and for any $f\in c-\Ind^H_C(V)$ there are  only finitely many elements $h_1,\ldots h_r$
	such that $\supp(f) = \bigcup_{i=1}^r h_i C.$
	Since $H$ is second countable we may write $c-\Ind^H_C(V)$ as a countable direct sum $\bigoplus_{g\in H/C} h \cdot V$ and supply this space with the locally convex direct sum topology.
	Then $c-\Ind ^H_{C}(V)$ is barrelled by \cite[Ex. 3 after 6.16]{S1} and Hausdorff by \cite[Cor. 5.4]{S1}. Further the action is locally analytic so that we get a locally analytic $H$-representation. The construction is functorial and we get a functor
	$$c-\Ind^H_C: \Rep_K^{la}(C) \to \Rep_K^{la}(H).$$

	As in the case of smooth representations we have:
	\begin{prop}
		\label{Mackeyadj}
		The functor $c-\Ind^H_C$ is left adjoint to the restriction functor $$\Rep(H)_K^{la} \to \Rep_K^{la}(C),$$ i.e, we have functorial bijections of $K$-vector spaces
		$$\Hom_H(c-\Ind^H_C(V),Z) \cong \Hom_C(V,Z_{|C}).$$ 
	\end{prop}
	
	\begin{proof}
		The proof is the same as in the case of smooth representations. Indeed, let $f \in \Hom_H(c-\Ind^H_C(V),Z).$ By composing it with the natural inclusion $V \hookrightarrow c-\Ind^H_C(V)$ we get a $C$-equivariant map $V \to Z$. On the other hand if $h \in  \Hom_C(V,Z_{|C})$ then we define $f: c-\Ind^H_C(V) \to Z$ by $f(\sum_g gv_g)=\sum_g gf(v_g).$ 
		The maps are inverse to each other.
	\end{proof}
	
	For later use we also consider a variant of the above construction which is common in smooth representation theory.
	Let $Z=Z(G)\subset G$ be the centre of $G.$ Then we define for any locally analytic $CZ$ -representation $V$
	$$c-\Ind^H_{ZC}(V):=\{f \in \Ind^H_{ZC}(V) \mid f \mbox{ has compact support modulo $Z(G)$} \}.$$ For more details, we refer to the appendix.
	
	The  representations $c-\Ind^H_C(V)$  and $c-\Ind^H_{ZC}(V)$ are of compact type by \cite[Prop 1.2. ii]{ST1} but in general not admissible.
	For this reason we enlarge the concept of "admissibility". 
	
	\begin{dfn}\label{dfn_qa}
		A locally analytic $H$-representation $V$ is called ind-(strongly )admissible if $V_{|C}=\varinjlim_n V_n$ is a strict inductive limit of (strongly) admissible locally analytic $C$-representations $V_n.$
	\end{dfn}

	The above definition does not depend on the chosen compact open subgroup  $C$ as in the case of a (strongly) admissible $H$-representation, cf.  Corollary \ref{indep} for a proof. Also if $W$ is a (strongly) admissible $C$-representation, then $c-\Ind^H_C(W)$ is ind-(strongly) admissible, cf. the proof of Proposition \ref{ind-stradm}.
	
	Clearly any (strongly) admissible $H$-representation is ind-(strongly )admissible. On the other hand, a kind of  converse is true.
	
	\begin{lemma}
		Every admissible representation is ind-strongly admissible.
	\end{lemma}
	
	\begin{proof}
		Let $V$ be an admissible $H$-representation. By definition the dual of its restriction to  to $C$  can be written as a countable projective limit of finitely generated modules over Banach algebras, i.e., $M:=V_C' =\varprojlim_{n\in \bbN} M_n.$ By \cite[Cor. 3.1]{ST2} we can suppose that the  finitely many generators of $M_n$ lift to elements $m_n^i,i\in I_n,$ of $M$. We consider the topological closure $T_n:=\overline{\langle m_n^i \mid i \in I_s, s\geq n \rangle}$  of the submodule generated by the above lifts which is a closed submodule of $M$. The quotient $M/T_n$ is thus by \cite[Lemma 3.6]{ST2} a coadmissible module which is generated by the finite set $\bigcup_{r < n} \{m^r_i \mid i\in I_r\}.$ Hence the dual  $V_n:=(M/T_n)'$ is strongly admissible and a closed subrepresentation of $V$. We get $V=\bigcup_n V_n$ with $V_m \subset V_n$ for $m<n$ as $M=\varprojlim_n M/T_n.$ 
	\end{proof}
	
	It follows that  every ind-admissible representation is ind-strongly admissible. In the  sequel we often say simply ind-admissible for convenience.

	\begin{lemma}
		\label{iacompact}
		Let $V=\varinjlim_n V_n$ be an ind-admissible representation. Then $V_n$ is a closed subrepresentation of $V_{|C}$ for all $n\in \bbN.$ Moreover, $V$ is of compact type. In particular it is Hausdorff, complete, bornological and reflexive.
	\end{lemma}
	
	\begin{proof}
		The first statement follows from \cite[Prop. 5.5 iii)]{S1}.
		The last assertion is the content of \cite[Theorem 1.1]{ST1}.
		The second statement follows from \cite[Theorem 1.2]{ST1} since any admissible representation is of compact type by definition.
	\end{proof}
	
	\begin{prop}\label{proj_lim_Frechet}
		Let $V=\varinjlim_n V_n$ be an ind-admissible $H$-representation. Then $V'_b=\varprojlim_n (V_n)'_b$ is the locally convex projective limit of the Fr\'echet spaces  ($V_n)'_b$. In particular $V'_b$ is a Fr\'echet space.
	\end{prop}
	
	\begin{proof}
		We follow the proof of \cite[Prop. 16.10]{S1} dealing with compact inductive limits.  We let $\Psi: V'_b \to \varprojlim_n (V_n)'_b$ be the natural map. Then $\Psi$ is clearly injective since $V=\bigcup_n V_n.$
		On the other hand, let $(l_n)_n$ be an element of $\varprojlim_n (V_n)'_b$. We define $l\in V'$ by
		$l(v):=l_i(v_i)$ if $v\in V_i.$ This construction gives rise to a well-defined element of the algebraic dual of $V.$ But it follows by the topology on $V$, that it is continuous , as  well. Indeed we follow the argument as in the proof of \cite[5.5 i)]{S1}. Let $L_i\subset V_i$  be a lattice which is in the preimage of of some fixed lattice $M\subset K$ under $l_i.$ Then we may construct inductively lattices $L_n, n\geq i$ such that that  $L_{n+1}\cap V_n\subset L_n$ and $L_n$ is in the preimage of the lattice $M$ under $l_n.$ Then $L:=\sum_i L_i$ is a lattice in $V$ which lies in the preimage of $M$ under $l.$     
		
		The map $\Psi$ is continuous since the maps $V'_b \to (V_n)'_b,  n\in \bbN,$ are continuous and by the very definition of the initial topology on  the right hand side.
		
		The space $V$ is bornological by \cite[Ex. 2 after Prop. 6.13]
		{S1}. Moreover, for each bounded subset $B\subset V$ there is by \cite[Prop. 5.6]{S1} some $i\in \bbN$ such that $B\subset V_i$ is a bounded subset. Since each $V_i$ is a countable limit of Banach spaces by \cite[Prop 6.5]{ST2} it follows by using Prop. 6.19 ii) and the argumentation in the proof of \cite[Prop. 16.10]{S1} that the topology of $V'$ can be defined by a countable limit of lattices. Hence it is by \cite[Prop. 8.1]{S1} metrizable and thus a Fr\'echet space by \cite[Prop. 9.1]{S1} since V is bornological. 
		
		On the other hand, the right hand side is also Fr\'echet space since it is the projective limit of Fr\'echet spaces (e.g. use \cite[Prop 8.1]{S1}).  It follows by the open mapping \cite[Prop. 8.6]{S1} theorem that $\Psi$ is a topological isomorphism.
	\end{proof}
	
	We continue with another definition which might be useful.
	
	\begin{dfn}
		A locally analytic $H$-representation $V$ is called ind-fJH if $V_{|C}=\varinjlim_n V_n$ is a strict inductive limit of locally analytic $C$-representations $V_n$ which admit a finite Jordan-Hölder series.
	\end{dfn}
	
	Clearly any locally analytic  representation which has a finite Jordan-Hölder series is strongly admissible. Hence any ind-fJH representation is ind-(strongly )admissible. 

	\begin{prop}\label{factorizes}
		Let $W=\varinjlim_n W_n$ be an ind-admissible $H$-representation. Let $f:V\to W$ be continuous map of $C$-representations where $V$ is locally analytic $C$-representation which has a finite JH-series. Then $f$ factorizes over some subspace $W_n.$
	\end{prop}
	
	\begin{proof}
		We claim that the image $f(V)$ is closed in $W.$ Indeed write
		$V=\bigcup_n f^{-1}(W_n)$. Then each subspace $V_n:=f^{-1}(W_n)$ is closed in $V$ as $f$ is continuous and $W_n\subset W$ is closed. Since closed subrepresentations of admissible representations are admissible again \cite[Prop. 6.4]{ST2} we see that the restriction $f_n: V_n \to W_n$ is a homomorphism of admissible $C$-representations. As such the image $f_n(V_n)=f(V_n)$ is closed in $W_n$ by loc.cit. We apply \cite[Remark 7.1 v) and Lemma 7.9]{S1} to deduce that 
		$f(V)=\bigcup_n f(V_n)$ is closed in $W$.
		It follows that $f(V)\cong V/ker(f)$  as admissible representations by \cite[Prop. 6.4]{ST2}.
		
		For the last assertion we may suppose that $V\subset W.$ Let $W_m$ be a subspace with $W_m\cap V\neq (0)$.  If $V$ is topologically irreducible it follows that $V\subset W_m.$ Otherwise, one argues by the number of irreducible constituents of $V$ that there must be an index $n$ with $V\subset W_n.$ 
	\end{proof}
	
	We denote by $\Rep^{ind-adm}_K(H)$ the full subcategory of $\Rep^{la}_K(H)$ consisting of ind-admissible objects. The full subcategory consisting of objects which are ind-fJH is denoted by $\Rep^{ind-fJH}_K(H).$ 
	
	\begin{prop}
		\label{abeliancat}
		The categories $\Rep^{ind-adm}_K(H)$ and $\Rep^{ind-fJH}_K(H)$ are abelian.
	\end{prop}

	\begin{proof}
		Let $f:V\to W$ be a homomorphism of ind-admissible representations. Write as before  $V=\varinjlim_n V_n$ and $W=\varinjlim_n W_n.$ We suppose first that $V_n$ and $W_n$ have a finite Jordan-Hölder series for all $n.$ Then by Proposition \ref{factorizes} for every $n$ there is some $\tau(n)\in \bbN$ such that $f(V_n)\subset W_{\tau(n)}.$ It follows by the same reasoning as above that $f(V)=\bigcup_n f(V_n)$ is closed in $W.$
		In particular it follows that ${\rm coker}(f)=W/\im(f)=\varinjlim_n W_n+\im(f)/\im(f)=\varinjlim_n W_n/\im(f)\cap W_n$ is ind-fJH again. Further $\ker(f)\cap V_n$ is closed in $V_n$ and has a finite Jordan-Hölder series.
		Hence   $\ker(f)=\varinjlim_n \ker(f)\cap V_n$ is closed in $V$  and ind-fJH. Finally, we prove that $f$ is strict.  Any closed subspace $Y$ of $V$ is of the shape $\varinjlim_n Y_n$ with $Y_n \subset V_n$ closed. As the induced maps $f_n: V_n \to W_{\tau(n)}$ are all strict their images $f_n(Y_n)=f(Y_n)$ are closed in $W_{\tau(n)}.$ We get by the above reasoning that $f(Y)$ is closed in $W$ and hence in $f(V).$ The other properties of an abelian category are checked easily.
		
		If $V_n$ and $W_n$ are admissible for all $n$, then we deduce in the same way as above that  $f(V_n)$ is admissible again. Hence $W_n+f(V_n)$ is admissible, as well, and we may replace $W_n$ by $W_n+f(V_n)$ for all $n.$ Then we may assume that $f(V_n)\subset W_n$ for all $n$ and proceed as before.
	\end{proof}

	\vspace{1cm}
	\section{Cuspidal Lie algebra representations}
	In this section we recall the theory of cuspidal modules of a Lie algebra over the field of complex numbers $\bbC.$ We thus replace here our $p$-adic field $L$ by the field of complex numbers $\bbC.$
	
	Recall that $T\subset G$ is a maximal torus and 
	$B\subset G$ is a Borel subgroup with $T \subset B.$ We denote by $\Phi \subset \frt^\ast$ the attached root system and by $\Phi^+$ resp. $\Phi^-$ its subset of positive resp. negative roots.
	Let $\rho=\frac{1}{2} \sum_{\alpha \in \Phi^+}\alpha\in \frt^\ast$ and denote by $Q\subset \frt^\ast$ the root lattice. For every $\alpha \in \Phi$ let $x_\alpha$ be a standard generator of the corresponding one-dimensional root space.
	
	A weight module for $\frg$ is a $U(\frg)$-module $M$ such that $\frt$ acts semi-simply on $M$ with finite-dimensional weight spaces, i.e.,
	$$M= \bigoplus_{\chi \in \frt^\ast} M_\chi$$
	with $M_\chi=\{m\in M\mid tm=\chi(t)m\, \forall t\in \frt\}$ and $\dim M_\chi < \infty$ for all $\chi$. The weight modules form a full subcategory $\cW$ of the category of all Lie algebra representations. 
	
	let $Z(\frg)\subset U(\frg)$ be the centre of $\frg.$ For $\lambda \in \frt^\ast$ let $\chi_\lambda: Z(\frg) \to K$ be the attached central character \cite[1.7]{H1}.  A central character $\chi_\lambda$ is called regular if the stabiliser of $\lambda +\rho$ is trivial. Otherwise it is called singular. It is called integral if $\lambda \in \Lambda$ where $\Lambda \subset \frt^\ast $ is the weight lattice. 
	
	By a result of  Fernando \cite[Thm. 4.18]{Fe} every simple weight module is a quotient of a parabolically induced  representation $U(\frg) \otimes_{U(\frp)} W$ for some simple weight module $W$ of the Levi component in $\frp.$  A simple weight module $M$ is called cuspidal if it does not admit such a representation as a quotient with $\frp \not\subset \frg.$ By loc.cit. this is equivalent to the property that all $x_\alpha$ with $\alpha \in \Phi$ act injectively on $M$. Moreover, in this case there exists an integer $d\geq 1$ such that $\dim M_\chi=d$ for all $\chi$ with $M_\chi \neq 0$, cf. \cite[Cor. 1.5]{Ma}. In particular $x_\alpha$ acts bijectively on $M$ for all $\alpha \in \Phi.$  We call the integer $d$ the degree of $M$.
	Again by a result of Fernando \cite[Thm. 5.2]{Fe} for simple Lie algebras $\frg$  cuspidal representations only exists for $\frg$ of type $A_n$ or $C_n.$ 
	
	We extend the definition  of cuspidality as in the paper \cite{GS} of Grantcharov and Serganova to non-simple modules.
	
	\begin{dfn}
		A weight module $M \in \cW$ is called cuspidal if the multiplication maps $M_\chi \stackrel{x_\alpha}{\to} M_{\chi+\alpha}$ are bijective for all $\alpha \in \Phi$ and for all $\chi\in \frt^\ast$.
	\end{dfn}
	
	\begin{rmk}
		Such cuspidal modules are sometimes also called  torsion free modules, cf. \cite[Remark below  Cor. 1.4]{Ma}. By definition the latter modules are those weight modules such that elements of $\frg\setminus \frt$ act bijectively on $M$, cf. \cite{BKLM}.
	\end{rmk}

	We let $\cC$ be the category of cuspidal $\frg$-modules. 
	This is an abelian category where each object admits a finite Jordan-H\"older series.
	
	In the case of ${\bf G}={\bf \GL_{n+1}}$, or rather ${\bf G}={\bf {\rm SL}_{n+1}}$, we can describe these modules explicitly. We follow the construction of Britten, Lemire \cite{BL} and Grantcharov and Serganova \cite{GS}, respectively. We start with degree 1 cuspidal modules. We suppose that $\frt$ is the diagonal torus in $\frg$ and identify its dual space $\frt^\ast$ with $\bbC^{n+1}$ in the usual way.  
	
	Let $\mu=(\mu_0,\ldots,\mu_n) \in \mathbb{C}^{n+1}$.  
	We set as in \cite{GS} 
	$$|\mu|=\mu_1+\cdots + \mu_n \in \mathbb{C}  \mbox { and } t^\mu:=t_0^{\mu_0}\cdots t_n^{\mu_n}$$ for the multivariable
	$t=(t_0,\ldots,t_n)$ \footnote{considered as meromorphic functions or as formal symbols with the same arithmetic properties.}.
	Set 
	$$M^\mu:=\{f \in t^\mu \mathbb{C} [t_0^{\pm 1}, \ldots, t_n^{\pm 1}] \mid |\mu| f=E f \}.$$
	Here $E$ is the differential operator defined as $E=\sum_i t_i \partial/\partial t_i $. The action of $x_\alpha$ with $\alpha=\epsilon_i-\epsilon_j, i\neq j,$ is given by applying the operator  $t_i\cdot \partial/\partial t_j$. The action of $t=(x_0,\ldots,x_n)\in \frt$ on $t_0^{\lambda_0}\cdots t_n^{\lambda_n}$ is given by multiplication with $\sum_i x_i \lambda_i.$
	
	\begin{lemma}
		Suppose that $\mu_i \not\in \bbZ$ for all $i$. Then $M^\mu$ is simple cuspidal. Moreover,  $M^\mu = M^\nu$ if $\mu -\nu \in Q$. 
	\end{lemma}

	\begin{proof}
		The proof is an easy exercise. Cf. also \cite[Thm. 2.2, Thm. 1.8]{BL}, \cite[p. 5]{GS}.
	\end{proof}
	
	In order to introduce cuspidal modules of higher degree we let $P=P_{(1,n)}$ be the maximal standard parabolic subgroup of $G$ attached to the decomposition $(1,n).$ Let $Z$ be additionally an algebraic $P$-representation. Thus we get a homogeneous vector bundle $\cE_Z$ on projective space $\bbP^n_{\bbC}$ and therefore an action of $\frg$  on $\cE_Z(D_+(t_0))\cong \cO_{\bbP^n}(D_+(t_0))\otimes_{\bbC} Z.$   Then we set 
	$$M^\mu(Z):= M^\mu \otimes_{\bbC} Z = M^\mu \otimes_{\cO_{\bbP^n}(D_+(t_0))} \cE_Z(D_+(t_0))$$
	and supply this with the diagonal $\frg$-action on the right hand side.
	This is again a cuspidal $U(\frg)$-module but in general not simple.
	For every exact sequence $0 \to Z_1 \to Z_2 \to Z_3 \to 0$, the induced sequence 
	\begin{equation}\label{exact_W}
		0 \to M^\mu(Z_1) \to M^\mu(Z_2) \to M^\mu(Z_3) \to 0  
	\end{equation}
	is exact, as well.

	We recall the following theorem of Matthieu, cf. \cite[Thm.2.5]{GS}.
	
	\begin{thm}
		Let ${\bf G}={\bf SL_{n+1}}.$
		Let $M$ be a cuspidal $\frg$-module and suppose that its central character $\chi_M$ is non-integral or singular integral. Then there is some simple finite-dimensional algebraic $P$-representation $Z$ such that $M\cong M^\mu(Z).$ 
	\end{thm}

	\begin{rmk}
		i) By \cite[Lemma 2.3]{GS} the central character of $M^\mu(Z)$ is given by $\chi_{\gamma(|\mu|\epsilon_0) +\tau}$
		where $\epsilon_0, \epsilon_1,\ldots,\epsilon_n   \in \bbC^{n+1}$ is the standard basis and $\gamma: \bbC^{n+1} \to \frt^\ast$ is the projection with kernel generated by $\epsilon_0 + \epsilon_i +\cdots +\epsilon_n.$ Further $\tau$ is the highest weight of $Z.$
		
		ii) There is also a description of simple cuspidal modules with non-singular and integral weights \cite{GS}. We will address this case in an upcoming paper.
	\end{rmk}
	
	\vspace{1cm}
	\section{The functor $\cF^G_\kappa$}
	
	Now we come back to our $p$-adic situation. We replace
	$\mathbb{C}$ by our coefficient field $K$ and consider the exponential series $\exp(X)=\sum_i X^n/n!.$ Its disc of  convergence is the open disc  $D_{p^{-1/p-1}}(0)=\{x\in K \mid |x|< p^{-1/p-1}\}.$  Further let $\log(X)=\sum(-1)^{n+1}X^n/n$ the logarithm series which converges on the open disc $D_1(1):=\{x\in K \mid |x-1|<1\}.$ 
	
	\begin{lemma}
		For each $\mu \in K$ there is some locally analytic character $\chi_\mu:  O_K^\times \to J^\times$ for some finite field extension $J$ of $K$ such that $\chi_\mu(x)=\exp(\mu \log (x))$ for $|x-1| \ll \epsilon.$
	\end{lemma}
	
	\begin{proof}
		Let $q$ be the number of elements of the residue field of $O_K.$ Write $O_K^\times= R \times D_1(1)$ where $R$ is the set of $q-1$-roots of unity in $K$.  
		We consider as in \cite[Thm 47.10]{Schik} the power series $x^\mu:=\sum_{n=0}^\infty { \mu \choose n} (x-1)^n$ which converges on its disc $D$ of convergence depending of course on $\mu.$ We have  $D\subset D_1(1).$ Now the group $D_1(1)$ is finitely generated over its subgroup $D.$  Hence by the constructive proof \cite[Prop. 45.6]{Schik} we can extend the  function $x^\mu$ to a locally analytic character $\chi_\mu: D_1(1) \to J^\times$ for some finite field extension $J$ of $K$. On the subgroup $R$ the character $\chi_\mu$ is extended trivially.    
	\end{proof}
	
	\begin{rmk} The locally analytic character with the above property is not uniquely determined. 
	\end{rmk}
	
	We fix for $\mu\in K$ a locally character $\chi_\mu:O_K^\times \to J^\times$ as above and write also $t^\mu$ for $\chi_\mu(t)$. Then 
	let $\mu=(\mu_0,\ldots,\mu_n) \in K^{n+1}$. We set 
	again $|\mu|=\mu_0+\cdots + \mu_n \in K,$
	$t^\mu:=t_0^{\mu_0}\cdots t_n^{\mu_n}$ and 
	$$M^\mu:=\{f \in t^\mu J[t_0^{\pm 1}, \cdots, t_n^{\pm 1}] \mid |\mu|f=E f \}.$$
	Here $E$ is again the operator defined as $E=\sum_i t_i \partial/\partial  t_i $. 
	
	Since the character $\chi_\mu$ is a locally analytic character of $T(O_K)$ we get by scalar restriction from $T(O_K)$ to $T_0=T(O_L)$ a locally analytic action of $T_0$ on $M^\mu$
	in the sense of \cite{OS3}\footnote{In loc.cit. we have considered only objects in the category $\cO$ and a parabolic subgroup. This definition extends obviously to our situation.}. Let $Z$ be an algebraic $P_{(1,n)}$-representation.  Let $V$ be additionally a smooth $T_0$-representation. Then we get a locally analytic $T_0$-action on $M^\mu(Z) \otimes_K V'$  and thus via the trivial action $\frg$ on $V$ a separately continuous $D(\frg,T_0)$-module structure on $M^\mu(Z)\otimes_K V'.$ 
	Here we denote by $D(\frg,T_0)\subset D(G_0)$ the subalgebra generated by $\frg$ and $D(T_0).$ 
	We set $$\cF_\mu(Z,V):={\rm c-\Ind}^G_{G_0}\big((D(G_0) \otimes_{D(\frg,T_0)}  (M^\mu(Z)\otimes_K V'))'\big).$$
	This construction is functorial in each entry. Hence for a fixed $\mu$ we get thus a bifunctor
	$$\cF_\mu: \Rep_K^\alg(P_{(1,n)}) \times \Rep_K^\infty (T_0)\to \Rep^{la}_J (G).$$
	
	\vspace{0.5cm}
	\begin{prop}
		The functor $\cF_\mu$ is bi-exact.
	\end{prop}
	
	Instead of proving this statement we extend our functor by considering arbitrary weight modules instead  of $M^\mu$  since this is needed for the proof anyway. We can extend the above functor in fact to weight modules as long as the weights are bounded. Consider the set $K^{n+1}/\bbZ^{n+1}$ and choose for any $\bar{\mu} \in K^{n+1}/\bbZ^{n+1}$ a representative $\mu\in K^{n+1}$ and a locally analytic character $\chi_\mu$ as  above. By identifying the set $\bbZ^{n+1}$ in the  usual way with  algebraic characters on $T$ we get for any element in $\mu+\nu\in \bar{\mu}$ a locally analytic character $\chi_{\mu+\nu}:=\chi_\mu\cdot \chi_\nu$.
	In this way we have declared  for each element in $K^{n+1}$  a locally analytic character of $T(O_K).$ 
	
	We consider for a real number $\kappa \geq 1$, the subcategory $\cW_{\leq \kappa}$ consisting of objects whose weights $\lambda\in K^{n+1}$ are bounded by $\kappa$, i.e., such that $|\lambda_i|\leq \kappa$ for all $0\leq i\leq n.$ 
	Hence all characters $\chi_\mu$ appearing as weights have a common $p$-adic extension field $J$ as the target coefficient field. We can thus generalize the above construction by defining for every $M \in \cW_\kappa$, $Z\in \Rep_K^{alg}(P_{(1,n)}) $ and $V\in \Rep_K^{\infty}(T_0),$
	$$X=X({M,Z,V}):=D(G_0) \otimes_{D(\frg,T_0)}  (M(Z)\otimes_K V').$$ This is a separately continuous $D(G_0)$-module\footnote{Here $M(Z)$ is defined as above.} on a nuclear Fr\'echet space. 
	Again we set
	$$\cF^G_\kappa(M,Z,V):= {\rm c-\Ind}^G_{G_0}\big((D(G_0) \otimes_{D(\frg,T_0)}  (M(Z)\otimes_K V'))'\big).$$ 
	Thus we obtain a tri-functor
	$$\cF^G_\kappa: \cW_\kappa \times \Rep_K^\alg(P) \times \Rep_K^\infty (T_0)\to \Rep^{la}_J (G).$$
	
	\begin{prop}
		The functor $\cF^G_\kappa$ is exact in each entry.
	\end{prop}
	
	\begin{proof}
		Since ${\rm c-\Ind}^G_{G_0}$ and $M(\cdot )$ are exact functors by (\ref{exact_W}) it is enough to see that the expression $X({M,Z,V})$ defines an exact functor in $M$ and $V.$ Here the approach is as in \cite[Prop. 4.9 a)]{OS2} resp. \cite[Prop. 3.7]{OS3}. As for the exactness in the smooth entry, we choose a locally analytic section $s$ of the projection $G_0 \to G_0/T$ and set $\cH=s(G_0/T_0).$  Then $X({M,Z,V})= D(\cH)\otimes_{U(\frg)} M \otimes_K V'$ as vector spaces.
		
		As for the exactness in $M$ the  proof is the same as in \cite[Prop. 3.7]{OS3} replacing the parabolic subgroup $P$ by $T$ and Verma modules by free weight modules $M(\lambda):=U(\frg)\otimes_{U(\frt)} K_\lambda$ with $\lambda\in K^{n+1}$.
		For such free weight modules we have $D(G_0)\otimes_{D(T_0,\frg)} M(\lambda) =  D(G_0) \otimes_{D(T_0)} J_\lambda.$ Then we let $I\subset G_0$ be the Iwahori subgroup
		and write $D(G_0) \otimes_{D(T_0)} J_\lambda= \bigoplus_{g\in G_0/I} \delta_g\cdot D(I) \otimes_{D(T_0)} J_\lambda.$ 
		Thus we may replace for proving the exactness $G_0$ by $I.$ Then we consider the Iwahori decomposition $I=(U_B^- \cap I)(T_0 \cap I) (U_B \cap I)=(U_B^- \cap I)T_0 U_B.$ Hence we may write $D(I) \otimes_{D(T_0)} J_\lambda = D((U_B^-)\cap I) \otimes_J D(U_B) \otimes  J_\mu).$ Then we proceed as in loc.cit for the remaining argument considering the tensor product $D((U_B^-)\cap I) \otimes_J D(U_B)$ (In loc.cit.  only the factor $D(U_B^-)$ is treated but the generalisation to the tensor product works in the same way.) 
	\end{proof}
	
	\begin{rmk}
		Alternatively, one may use for the second part of the above proof a recent result of Agrawal and Strauch \cite[Cor. 7.8.7]{AS}  which applies to more general closed subgroups of $G_0.$
	\end{rmk}

	Let $U\subset M(Z)$ be a finite-dimensional $T_0$-subspace generating $M(Z)$ as a $U(\frg)$-module. 
	Then we may identify similarly to \cite{OS2} our representation $\cF^G_\kappa(M,Z,V)$ with the closed subrepresentation 
	$$c-\Ind^G_{T_0}(U'\otimes Z)^\frd:=\{f\in c-\Ind^G_{T_0}(U'\otimes V) \mid \langle x,f \rangle=0 \, \forall x\in \frd  \}$$ of $c-\Ind^G_{T_0}(U' \otimes V)$ 
	where 
	\begin{equation}\label{pairing}
		\begin{array}{rccc}
			\langle \cdot , \cdot \rangle: & \left(D(G_0) \otimes_{D(T_0)}  U \right) \otimes_J \Ind^{G_0}_{T_0}(U' \otimes V) & \lra & C^{an}(G_0,V) .\\
			&&&\\
			& (\delta \otimes w) \otimes f & \mapsto & \Big[ g \mapsto \big(\delta \cdot_r (f(\cdot)(w))\big)(g)\Big]
		\end{array}
	\end{equation}
	\noindent Here, by definition, we have $\big(\delta \cdot_r (f(\cdot)(w))\big)(g) = \delta(x \mapsto f(gx)(w))$. 
	
	We denote by ${\bf 1}$ the trivial representation of $P$ and by ${\bf 1}^\infty$ the trivial (smooth) representation of $T_0$.
	If $Z$ or $V$ is the trivial representation then we simply omit it from the input in our functor as in \cite{OS2}. E.g., we write
	$\cF^G_\kappa(M)$ for  $\cF^G_\kappa(M,{\bf 1},{\bf 1}^\infty).$ 

	\begin{prop}
		Let $M\in \cW_\kappa$, $Z\in \Rep^{alg}_K(P_{(1,n)}), V\in \Rep^\infty_K(T_0)$. Suppose that $V$ is admissible.  Then the representation $X({M,Z,V})'$ is strongly admissible.
	\end{prop}
	
	\begin{proof}
		The Lie algebra representation $M$ is finitely generated. As $Z$ and $V$ are finite-dimensional the module $X({M,Z,V})=D(G_0) \otimes_{D(\frg,T_0)} M(Z)\otimes V'$ is finitely generated, as well.  
	\end{proof}

	\begin{prop}\label{ind-stradm}
		Let $M\in \cW_\kappa$, $Z\in \Rep^{alg}_K(P_{(1,n)}), V\in \Rep^\infty_K(T_0)$. Suppose that $V$ is admissible. Then
		the representation $\cF^G_\kappa(M,Z,V)$ is ind-admissible.
	\end{prop}
	
	\begin{proof}
		By the proposition before the representation $X'= X({M,Z,V})'$ is strongly admissible.  Let $S\subset G$ be a set of representatives for the double cosets $G_0 \backslash G/G_0.$ 
		As for representations of finite groups \cite[Proposition22]{Se} we have  
		\begin{eqnarray*}
			c-\Ind^G_{G_0}(X')_{|G_0} & = & \bigoplus_{s \in S}  \Ind^{G_0}_{G_0 \cap  sG_0s^{-1}} sX' .
		\end{eqnarray*}
		Now $G$ is second countable hence $G_0\backslash G/G_0$ is countable. Since any locally convex sum is a strict inductive limit and $G_0 \cap  sG_0s^{-1}$ is of finite index in $G_0$ the claim follows.  
	\end{proof}

	\vspace{1cm}
	\section{Irreducibility}
	
	In general we cannot expect the representation $\cF_\mu(Z,V)$ to be topologically irreducible even if $Z$ and $V$ are irreducible due to the action of the centre $Z(G)$ of $G$. For this reason we consider the following variant of $\cF_\mu.$
	Suppose that $Z$ and $V$ are irreducible. Then the action of $Z(G)\cap T_0$ on $(M^\mu(Z)\otimes_K V')'$ is given by a locally analytic character $\eta_{\mu,Z,V}$. We fix a locally analytic character $\eta$ on $Z(G)$ which extends $\eta_{\mu,Z,V}.$ Then we set
	$$\overline{X}_\mu(Z,V):=D(Z(G)G_0) \otimes_{D(\frg,T_0)}  (M^\mu(Z)\otimes_K V')\otimes J_\eta'$$
	and
	$$\overline{\cF}_\mu(Z,V):={\rm c-\Ind}^G_{Z(G)G_0}(\overline{X}_\mu(Z,V)').$$

	We denote by $W$ the Weyl group of $G.$ For $n=1,$ the parabolic subgroup $P_{(1,n)}$ is just the Borel subgroup $B$ and any irreducible algebraic representation $Z$ of $B$ is given by an algebraic character $\lambda\in X^\ast(T)$ of the maximal torus $T.$ Thus we may write $\lambda=(\lambda_0,\lambda_1)$ with $\lambda_i\in \bbZ,i=0,1.$
	
	\begin{thm} \label{irred} Let $n=1$. Let $\mu \in K^2$ and let $\lambda \in \Rep^{alg}_K(T)$, $V\in \Rep^\infty_K(T_0)$   be simple objects. 
		Then  $\overline{X}_\mu(Z,V)'$ is topologically irreducible if $w(\mu)- \mu \not\in \bbZ^{2}$ for the non-trivial reflection $w \in W$, $|\mu_i|\notin |L|,i=0,1,$ $|\mu_0|\neq |\mu_1|$ and 
		
		a) $|\mu_0| \leq  1$ or $|\mu_1| \leq  1$ if $\lambda_0+\lambda_1=0$.
		
		b) $|\mu_0| \leq  1$ and  $|\mu_1| \geq  1$ or $|\mu_0| \geq  1$ and  $|\mu_1| \leq 1$ if $\lambda_0+\lambda_1\neq 0$.
	\end{thm}

	\begin{rmk}
		It is tempting to generalize the above statement to arbitrary $n$ in form of  a conjecture. At least the condition 
		$w(\mu)- \mu \not\in \bbZ^{n+1}$ for all $w \in W, w\neq 1$
		seems to be necessary as we see from the proof below.
	\end{rmk}

	\begin{proof}
		We abbreviate $M^\mu$ by $M.$  If $\lambda_0+ \lambda_1=0$ then $\lambda$  is absorbed by  $J[t_0^{\pm 1}, t_1^{\pm 1}]$ and hence $M^\mu=M^\mu(Z)$ so that we can assume that $Z$ is the trivial representation. In general, the character $\mu':=\mu+\lambda$  satisfies again the assumptions $w(\mu')-\mu' \not\in \bbZ^2.$ Suppose that $\mu_0<1$ and $\mu_1>1$ Then $|\mu_0'|\geq  1$ and $|\mu_1'|=|\mu_1|.$ By adding the character $t_0^{-m_0}t_1^{m_0}\in J[t_0^{\pm 1}, t_1^{\pm 1}]$ to $\mu'$ we get a new character $\mu^{''}$ with $\mu_0=\mu^{''}_0$ and that the other conditions are still satisfied. So the upshot is that we may assume that $Z={\bf 1}.$ The case $|\mu_0|>1$ and $|\mu_1|<1$ is symmetric.

		As for $V$ it is given by a smooth character $\chi:T_0 \to J^\times.$ It is enough to check  that $X$ is a simple $D(G_0)$-module. Here we mimic the proof of \cite[Thm 5.3]{OS2} dealing with principal series representations and use the notation there.  
		The proof is at the beginning verbatim the same as in loc.cit.  by replacing $P_0$ by $T_0$ and with the obvious (notational) changes. Here we consider as mentioned above arbitrary $n$ and specialize later in the proof to $n=1.$ We go through the proof of loc.cit. and recollect the main constructions in our modified version. But there is one difference. In contrast to loc.cit. we include the smooth  character in the proof and do not handle it separately since it induces only a twist of the objects therein. Thus we denote by $M(\chi)=M\otimes \chi'$ the $D(\frg,T_0)$-module with the trivial $\frg$-action on $\chi$ and the tensor product action of $T_0$. 
		
		We consider thus for an arbitrary open normal subgroup $H$  of $G_0$ the decomposition $$X=\bigoplus_{g\in G_0/H} \delta_g \star \left(D(HT_0) \otimes_{D(\frg,T_0)} M(\chi) \right).$$ Here we recall  that for a $D(H)$-module $N$ and $g \in G_0,$ we denote by $\delta_g \star N$ the space $N$, equipped with the structure of a $D(H)$-module given by $\delta \cdot_g n = (\delta_{g^{-1}} \delta \delta_g)n$, where $n \in N$, $\delta \in D(H)$, and the product $\delta_{g^{-1}} \delta \delta_g \in D(G_0)$ is contained in $D(H)$.
		
		We arrive at the following modified version of Thm. 5.5  loc.cit. in a reduced version which proves our theorem.
		
		\begin{thm}\label{irredH}  Let $H$ be an open normal subgroup of $G_0$, and let $g, g_1, g_2 \in G_0$. Then
			
			(i) The $D(H)$-module $\delta_g \star \left(D(HT_0) \otimes_{D(\frg,T_0)} M(\chi)\right)$ is simple.
			
			(ii) The $D(H)$-modules $\delta_{g_1} \star \left(D(HT_0) \otimes_{D(\frg,T_0)} M(\chi)\right)$ and $\delta_{g_2} \star \left(D(HT_0) \otimes_{D(\frg,T_0)} M(\chi)\right)$ are isomorphic if and only if $g_1 HT_0 = g_2 HT_0$.
		\end{thm}
		
		\noindent \Pf The proof passes through several reduction steps. The first step is to reduce to the case of a suitable subgroup $H_0 \sub H$ which is normal in $G_0$ and uniform pro-$p$.  
		
		{\it Step 1: reduction to $H_0$.}  Let $\Lie(\bG_0)$, $\Lie(\bT_0)$ etc.  be the Lie algebras of $G_0$, $T_0$ etc. These are $O_L$-lattices in $\frg = \Lie(G)$, $\frt = \Lie(T)$ etc. respectively. Moreover, $\Lie(\bG_0)$, $\Lie(\bT_0)$ etc.  are $\Zp$-Lie algebras, and we have as usual a decomposition $\Lie(\bG_0)= \Lie(\bU_{B,0}) \oplus \Lie(\bT_0) \oplus  \Lie(\bU^-_{B,0}). $
		
		\noindent For $m_0 \ge 1$ ($m_0 \ge 2$ if $p=2$) the $O_L$-lattices $p^{m_0} \Lie(\bG_0)$, $p^{m_0} \Lie(\bT_0)$ and $p^{m_0} \Lie(\bU_{\bP,0})$ are powerful $\Zp$-Lie algebras, cf. \cite[sec. 9.4]{DDMS}, and hence
		$\exp_G: \frg \dashrightarrow G$ converges on these $O_L$-lattices. We set 
		\begin{numequation}\label{H_0}
			H_0 = \exp_G\Big(p^{m_0} \Lie(\bG_0)\Big) \;, \;\; H_0^0 = \exp_G\Big(p^{m_0} \Lie(\bT_0)\Big), 
		\end{numequation}
		\begin{equation*} H_0^+ = \exp_G\Big(p^{m_0} \Lie(\bU_{B,0})\Big), H_0^- = \exp_G\Big(p^{m_0} \Lie(\bU_{B,0}^-)\Big),  
		\end{equation*}
		which are uniform pro-$p$ groups. Moreover $H_0$ is normal in $G_0$ and  $H_0^0$ in $T_0$, respectively. 
		We choose   $m_0$ large enough such that $H_0$ is contained $H$.
		With the same argument as in \cite{OS2} we may assume from now on that $H_0=H.$
		
		{\it Step 2: passage to $D_r(H)$.} We put
		$$\kappa = \left\{\begin{array}{lcl} 1 & , & p>2 \\
			2 & , & p=2 \end{array} \right.$$
		and denote by $r$ a real number in $(0,1) \cap p^\bbQ$ with the property that there is $m \in \bbZ_{\ge 0}$ such that $s = r^{p^m}$ satisfies\footnote{This assumption seems to be also necessary in \cite{OS2}}
		\begin{numequation}\label{r_and_s}
			\max\{\frac{1}{p},|\pi|p^{-1/(p-1)} \} < s \mbox{ and } s^{\kappa} < p^{-1/(p-1)} 
		\end{numequation}
		where $\pi\in O_L$ is an uniformizer.
		We let $\|\cdot\|_r$ denote the norm on $D(H)$ associated to the canonical $p$-valuation. Here $D_r(H)$ is the corresponding Banach space completion. This is a noetherian Banach algebra, and $D(H) = \varprojlim_{r <1} D_r(H)$. Let $\bM(\chi):=D(HT_0)\otimes_{D(\frg,T_0)}M(\chi).$ As explained in loc.cit. it is enough to show that
		$$\bM_r(\chi) := D_r(H) \otimes_{D(H)} \bM(\chi) = D_r(HT_0) \otimes_{D(\frg,T_0)} M(\chi)$$
		are simple $D_r(H)$-modules for a sequence of $r$'s converging to $1$. From now on we assume that, in addition to (\ref{r_and_s}), that $r$ is chosen such that $\bM_r(\chi) \neq 0$, and we consider $M(\chi)$ as being contained in $\bM_r(\chi)$.
		
		{\it Step 3: passage to $U_r(\frg)$.} Let $U_r(\frg) = U_r(\frg,H)$ be the topological closure of $U(\frg)$ in $D_r(H)$. Then
		\begin{numequation}\label{density1}
			\hskip-50pt \mbox{$U(\frg)$ is dense in $D_r(H)$ if $r^{\kappa} < p^{-\frac{1}{p-1}}$ and $U_r(\frg) = D_r(H)$.}
		\end{numequation}
		Let $(P_m(H))_{m \ge 1}$ be the lower $p$-series of $H$, cf. \cite[1.15]{DDMS}. Note that $P_1(H) = H$. For $m \ge 0$ put $H^m : = P_{m+1}(H)$ so that $H^0 = H$. Then $H^m$ is a uniform pro-$p$ group with $\Zp$-Lie algebra equal to $p^m\Lie_\Zp(H)$. 
		Thus our subgroups above have the following description.
		Let $\frg_\bbZ$ be a $\bbZ$-form of $\frg$, i.e. $\frg_\bbZ \otimes_\bbZ L = \frg$. We fix a Chevalley basis $(x_\gamma,y_\gamma,h_\alpha \midc \gamma \in \Phi^+, \alpha \in \Delta)$ of $[\frg_\bbZ,\frg_\bbZ]$. We have $x_\gamma \in \frg_\gamma$, $y_\gamma \in \frg_{-\gamma}$, and $h_\alpha = [x_\alpha,y_\alpha] \in \frt$, for $\alpha \in \Delta$. Then
		$$\Lie_\Zp(H^-) = p^{m_0}\Lie(\bU^-_{\bB,0}) = \bigoplus_{\beta \in \Phi^+} O_L y_\beta^{(0)} \;,$$
		where $y_\beta^{(0)} = p^{m_0}y_\beta$. Moreover, the $\Zp$-Lie algebra of $H^{-,m}$ is
		$$\Lie_\Zp(H^{-,m}) = p^m\Lie(H^-) = \bigoplus_{\beta \in \Phi^+} O_L y_\beta^{(m)} \,,$$
		where $y_\beta^{(m)} = p^{m_0+m}y_\beta$.
		In the same way
		$$\Lie_\Zp(H^{+,m}) = p^m\Lie(H^+) = \bigoplus_{\beta \in \Phi^+} O_L x_\beta^{(m)} \,,$$
		where $x_\beta^{(m)} = p^{m_0+m}x_\beta$.

		Let $s = r^{p^m}$ be as in (\ref{r_and_s}). Denote by $\|\cdot\|_s^{(m)}$ the norm on $D(H^m)$ induced by the canonical $p$-valuation on $H^m$. Then, by \cite[6.2, 6.4]{Sch}, the restriction of $\|\cdot\|_r$ on $D(H)$ to $D(H^m)$ is equivalent to $\|\cdot\|_s^{(m)}$, and $D_r(H)$ is a finite and free $D_s(H^m)$-module on a basis any set of coset representatives for $H/H^m$. By (\ref{density1}) we can conclude 
		\begin{numequation}\label{density2}
			\mbox{If $s = r^{p^m}$ is as in (\ref{r_and_s}), then $U(\frg)$ is dense in $D_s(H^m)$, hence $U_r(\frg,H) = D_s(H^m)$.}
		\end{numequation}
		In particular, $U_r(\frg) \cap H = H^m$ is an open subgroup of $H$. Let
		$$\frm_r(\chi) := U_r(\frg)M(\chi)$$
		\noindent be the $U_r(\frg)$-submodule of $\bM_r(\chi)$ generated by $M(\chi)$. The module $M(\chi)$ is dense in $\frm_r(\chi)$ with respect to this topology. It follows from \cite[1.3.12]{F}
		or \cite[3.4.8]{OS1} that $\frm_r(\chi)$ is a simple $U_r(\frg)$-module, and in particular a simple $D_r(\frg,T_0)$-module. Thus for every $g \in G_0$ the $U_r(\frg)$-module $\delta_g \star \frm_r(\chi)$ (which is defined as above by composing the natural action with conjugation with $g$) is simple. As in \cite{OS2} we have for $T_{0,r}:= HT_0 \cap D_r(\frg,T_0)$ 
		
		\medskip
		(i) $T_{0,r} = H^mT_0=H^{+,m}T_0 H^{-,m}$, where $H^{-,m} = P_m(H^-)$ and $H^{+,m} = P_m(H^+)$. 
		
		(ii) $D_r(HT_0) =  \bigoplus_{g \in HT_0/T_{0,r}} \delta_g D_r(\frg,T_0)$.
		
		\medskip
		By (ii) we have  $\frm_r(\chi) = D_r(\frg,T_0) \otimes_{D(\frg,T_0)} M(\chi) \sub D_r(HT_0) \otimes_{D(\frg,T_0)} M(\chi) = \bM_r(\chi)$ and 
		\begin{numequation}\label{decomp}
			\bM_r(\chi) = D_r(HT_0) \otimes_{D(\frg,T_0)} M(\chi) = D_r(HT_0) \otimes_{D_r(\frg,T_0)} \frm_r(\chi) = \bigoplus_{g \in HT_0/T_{0,r}} \delta_g \frm_r(\chi) \,.
		\end{numequation}
		Here the action of $U_r(\frg)$ on  $\delta_g \frm_r(\chi)$ is the same as on $\delta_g \star \frm_r(\chi).$
		Thus it suffices to show the theorem below.
		Here we only consider $\frm_r$ as a module over $U_r(\frg)$. This suffices to prove Theorem \ref{irredH}.  Therefore we omit $\chi$ for the remainder of this section.
		
		\begin{thm}\label{U_r_modules} Assume that all the  above assumptions are satisfied. Then
			for any $g_1, g_2 \in G_0$ with $g_1T_{0,r} \neq g_2T_{0,r}$ the $U_r(\frg)$-modules $\delta_{g_1} \star \frm_r$ and $\delta_{g_2} \star \frm_r$ are not isomorphic.
		\end{thm}

		\noindent \Pf We start with the following observation.
		By our assumption on $r$ and $s$ we have $U_r(\fru_{B}^-) = U_r(\fru^-_{B},H^-) = D_s(H^{-,m})$. Elements in $U_r(\fru^-_{B})$ thus have a description as power series in $(y^{(m)}_\beta)_{\beta \in \Phi^+}$:
		$$U_r(\fru_B^-) = \left\{\sum_{n = (n_\beta)} d_n  (y^{(m)})^n \midc \lim_{|n| \ra \infty} |d_n|s^{\kappa|n|} = 0 \right\} \,,$$
		where $(y^{(m)})^n$ is the product of the $(y^{(m)}_\beta)^{n_\beta}$ over all $\beta \in \Phi^+$, taken in some fixed order.
		Let $\|\cdot\|_s^{(m)}$ be the norm on $D_s(H^{-,m})$ induced by the canonical $p$-valuation on $H^{-,m}$. Then we have for any generator $y^{(m)}_\beta$
		\begin{numequation}\label{norm_generator}
			\| y^{(m)}_\beta\|_s^{(m)} = s^\kappa \;,
		\end{numequation}
		By symmetry the discussion above holds true for the group $H^{+,m}.$ Thus we may write
		\begin{align*}\frm_r & = & \\ & & \left\{\sum_{n = (n_\beta)} d_n  (y^{(m)})^n v_\mu + \sum_{n = (n_\beta)} c_n  (x^{(m)})^nv_\mu\midc \lim_{|n| \ra \infty} \|d_n (y^{(m)})^n\|_s^{(m)} = 0, \lim_{|n| \ra \infty} \|c_n (x^{(m)})^n\|_s^{(m)} = 0 \right\} \, 
		\end{align*}
		where $v_\mu\in M$ is a fixed generator of weight $\mu$.

		So let $\phi: \delta_{g} \star \frm_r \stackrel{\simeq}{\lra} \delta_{h} \star \frm_r$ be an isomorphism of $U_r(\frg)$-modules. This induces  a $U_r(\frg)$-module isomorphism $\delta_{h^{-1}g} \star \frm_r \stackrel{\simeq}{\lra} \frm_r$, so that we may assume $h = 1$. Let $I \sub G_0$ be the Iwahori subgroup.  Using the Bruhat decomposition 
		$$G_0 = \coprod_{w \in W} IwI$$
		we may write $g = k_1^{-1}wk_2$ with $k_1,k_2 \in I$, $w \in W$. By shifting $k_1$ from the left to the right and replacing $\frm_r$ by $\delta_{k_2^{-1}} \frm_r$ we get an isomorphism
		$$\phi: \delta_w \star \frm_r \stackrel{\simeq}{\lra} \delta_h \star \frm_r $$
		with $h \in I.$ 

		In the following, we write $x\cdot_g v$ for the action of $x\in \frg$ on a vector $v$ with respect to  the representation $\delta_g\star \frm_r.$  
		Now let $v_\mu \in M_\mu, v_\mu \neq 0.$  Then $v=\phi(v_\mu)$ is of weight $w(\mu)$ since $x\cdot_h v=\phi(x\cdot_w v_\mu)=\phi((w\mu(x))\cdot v_\mu)=(w\mu)(x)v.$
		But the weight vectors in $M^\mu$ are contained in $\mu +\bbZ^{n+1}$  and the weights for 
		$\delta_h\star M$ are the same as for $M$ i.e. $\mu+\bbZ^{n+1}.$
		Indeed write  $h=u^+ \cdot t\cdot  u^-$ as a product of an upper triangular unipotent element, an element $t\in T_0$ and a lower unipotent element. Here we may suppose that $t=1$ since it does not affect the weights.
		But $Ad(u^+)(x)$ has the shape $x+\sum_{\alpha>0} l_\alpha x_\alpha$ and $Ad(u^-)(x)$ has the shape $x+\sum_{\alpha<0} l_\alpha x_\alpha.$  Further any element $v \in M$ has the shape $v_\mu + \sum_\alpha v_{\mu+\alpha}.$ Now the claim follows by considering the products $x_\alpha v_\beta, \alpha,\beta \in \Phi$.   
		Since by assumption $w(\mu)- \mu \not\in \bbZ^{n+1}$ there cannot exist such an isomorphism $\phi.$
		
		\begin{rmk}
			Up to now the proof works for arbitrary $n.$
		\end{rmk}
		
		From now on $n$ is equal to $1.$ So we may suppose that $w=1$ and there is an isomorphism
		$$\phi: \frm_r \stackrel{\simeq}{\lra} \delta_h \star \frm_r \,$$
		with $h\in I.$ By the Iwahori decomposition
		$$I = (I \cap {U_{B}^-}) \cdot (I \cap B)$$
		we may write $h=u^-tu^+$ with $u^- \in  I \cap {U_{B}^-}, u^+ \in  I \cap {U_{B}^+}$ and $t \in I \cap {T_0} $. Since $t$ normalizes  the other groups and $\delta_t\star \frm_r=\frm_r$ we may assume that $t=1.$ Again we have to show that there cannot be such an isomorphism $\phi.$
		For this let again $v=\phi(v_\mu)$ with $v_\mu=t^\mu.$ Write
		\begin{numequation}\label{seriesv}
			v=\sum_{i\in \bbZ} c_i t^{\mu +i \alpha}
		\end{numequation}
		where $\alpha$ is the unique positive simple root and $c_i\in J.$ We shall prove that this formal series does not converge if $h\not\in T_{0,r}$. For doing so, we rewrite this series also in the shape
		\begin{numequation}\label{seriesvmitd}
			v=\sum_{i<0} d_i y_\alpha^{(i)} t^{\mu } + \sum_{i\geq 0} d_i x_\alpha^{(i)} t^{\mu }
		\end{numequation}
		since by (\ref{norm_generator}) this allows us to decide easier whether the power series converges.
		Then it follows that 
		$$c_i=\left\{\begin{array}{cc}
			\mu_0(\mu_0-1)\cdots (\mu_0+i+1)p^{-i(m_0+m)}d_i & i<0 \\ d_0 & i=0 \\ \mu_1(\mu_1-1)\cdots (\mu_1-i+1)p^{i(m_0+m)}d_i  & i>0 .
		\end{array} \right.$$
		As above we get by shifting the isomorphism
		$\delta_{u^+} \star \frm_r \cong \delta_{u^-} \star \frm_r$ and
		for all $v \in \frm_r , x\in \frg,$ the identity 
		$$\phi(Ad((u^{+})^{-1})(x)v)  = Ad((u^{-})^{-1})(x)\phi(v)
		.$$ 
		We shall examine what this means on the weight spaces and for the coefficients $c_i$, respectively. Thus we consider the identities
		$$\phi((Ad((u^{+})^{-1})(x)v_\mu))_\lambda  = (Ad((u^{-})^{-1})(x)\phi(v_\mu))_\lambda$$ where $\lambda \in \mu + \bbZ^{2}, x\in \frt$ and where $(\;)_\lambda$ indicates the $\lambda$ weight space. 
		
		Let $$e=\left( \begin{array}{cc}
			0 & 1 \\0 & 0
		\end{array} \right)
		,f=\left( \begin{array}{cc}
			0 & 0 \\1 & 0
		\end{array} \right), z=\left( \begin{array}{cc}
			1 & 0 \\0 & -1
		\end{array} \right)$$ so that $u^+=E_2+a e$ and $u^-=E_2+b f$ for some $a,b \in L.$ Note that we have
		$$Ad((u^{+})^{-1})(x)=x + \alpha (x) ae $$ and 
		$$Ad((u^{-})^{-1})(x)=x -  \alpha (x) bf . $$
		Hence we get
		\begin{eqnarray*}
			Ad((u^{-})^{-1})(x)\phi(v_\mu) & = & \phi(Ad((u^{+})^{-1})(x)v_\mu)=\phi((x + \alpha (x) ae)v_\mu)= \phi(xv_\mu+ \alpha(x)ae v_\mu)\\ & = &\mu(x)v+ \alpha(x)a\phi(ev_\mu).
		\end{eqnarray*}Since one checks that 
		$Ad((u^{+})^{-1})(e)=e$ the latter identifies with
		$$\mu(x)v+ \alpha(x) aAd((u^{-})^{-1})(e)v.$$
		But $Ad((u^{-})^{-1})(e)= -b^2f + bz + e$. Hence the latter expression coincides with
		$$\mu(x)v+ \alpha(x) a (-b^2f + bz + e)v.$$
		On the other hand,
		$Ad((u^{-})^{-1})(x)\phi(v_\mu)=(x -\alpha(x) b f)v=xv-\alpha(x)bfv.$ Thus we get
		\begin{numequation}\label{equationmu}
			\mu(x)v+ \alpha(x) a (bz + e)v = xv + \alpha(x) (b^2a-b)fv.
		\end{numequation}
		for all $x\in \frt.$
		
		Next we compare the weights in the expression above.
		We have for $\lambda=(\lambda_0,\lambda_1)\in J^2$ 
		$$ e t^\lambda = \lambda_1 t^{\lambda+\alpha}, f t^\lambda = \lambda_0 t^{\lambda- \alpha} $$
		and 
		$$ z t^\lambda = (\lambda_1-\lambda_0) t^{\lambda} .$$ 
		
		\vskip10pt
		\noindent {\bf First case:} $b=0, a\neq 0.$  Then the equation (\ref{equationmu}) is just
		$$ \mu(x)v+ \alpha(x) aev = xv .$$
		For $\lambda=\mu$ we get 
		$$ \mu(x)c_0t^\mu+ \alpha(x)a  c_{-1}(\mu_1+1)t^\mu = \mu(x)c_0t^\mu.$$
		Hence $c_{-1}a(\mu_1+1) = 0$. It follows that $c_{-1}=0.$
		
		\noindent For $\lambda=\mu-\alpha$ we get 
		$$ \mu(x)c_{-1}t^{\mu-\alpha}+ \alpha(x)a c_{-2}(\mu_1+2)t^{\mu-\alpha} = \mu(x)c_{-1}t^{\mu-\alpha}.$$ By the step before it follows that  $c_{-2}a(\mu_1+2) =0$ and thus  $c_{-2}=0.$
		Successively we see that $c_{-i}=0$ for all $i\in \bbN.$

		\noindent For $\lambda=\mu+\alpha$ we get
		$$\mu(x)c_1t^{\mu + \alpha}+\alpha(x) a \mu_1 c_0t^{\mu+\alpha} = (\mu+\alpha(x))c_1t^{\mu+\alpha}.$$
		Hence $a \mu_1 c_0 = c_1$.
		
		For $\lambda=\mu+2\alpha$ we get
		$$\mu(x)c_2t^{\mu + 2\alpha}+\alpha(x)a (\mu_1-1) c_1 t^{\mu+2\alpha} = (\mu+ 2 \alpha(x))c_2t^{\mu+2\alpha}.$$
		Hence $a(\mu_1-1)c_1=2 c_2$ thus $c_2=\frac{a(\mu_1-1)c_1}{2}.$ By substituting $c_1$ be the above expression we get
		$$c_2=\frac{a^2\mu_1(\mu_1-1)c_0}{2}.$$
		By iterating this process we get
		$$c_i=\frac{a^i\mu_1(\mu_1-1) \cdots (\mu_1-(i-1))c_0}{i!}.$$ 
		If $c_0=0$, then all coefficients $c_i$ vanish, so that $v=0.$
		If $c_0\neq 0$ we may suppose by scaling that $c_0=1.$
		Hence we see that the series (\ref{seriesv}) coincides with the series which interpolates the expression $(1+a)^{\mu_1}.$ By definition of the topology on $m_r$ we have rather to consider the power series (\ref{seriesvmitd}) for its convergence which is here given by
		$$\sum_{i\geq 0} \frac{a^i}{i!p^{i(m_0+m)}} (y^{(m)})^i.$$ Here we can argue again as in loc.cit. It is even simpler. Indeed, suppose that $u \not\in T_{0,r}$, i.e. $a\not\in p^{m_0+m}O_L.$ We may write $a=\frac{1}{\pi^k} p^{m_0+m}\alpha$ with $\alpha \in O_L^\times$ and $k\geq 1.$ Then  
		$\frac{a^i}{i!p^{i(m_0+m)}} = \alpha^i/i!(\pi^k)^i$. It follows that $$|d_i|\cdot ||y^{(m)}||^{(m)}_s= 1/|i!||(\pi^k)^i| s^i=\pi^{-i(k-1)} / |i!| \cdot (\frac{s}{|\pi|})^i.$$
		Since $\frac{s}{|\pi|}> p^{-\frac{1}{p-1}}$ (which is the convergence radius of the exponential function)  by assumption (\ref{r_and_s}) and $k\geq 1$ this sequence does not converge to zero.

		
		\vskip10pt
		\noindent  {\bf Second case:} $b\neq 0, a= 0.$ Then the equation (\ref{equationmu}) is just
		$$ \mu(x)v = xv -\alpha(x)bfv .$$
		
		\noindent  For $\lambda=\mu$ we get 
		$$ \mu(x)c_0t^\mu = \mu(x)c_0t^\mu - \alpha(x)bc_1(\mu_0+1)t^\mu.$$
		Hence $b c_{1}(\mu_0+1) = 0$. It follows that $c_{1}=0.$
		
		\noindent  For $\lambda=\mu+\alpha$ we get 
		$$ \mu(x)c_{1}t^{\mu+\alpha} = (\mu+\alpha)(x)c_{1}t^{\mu+\alpha}-\alpha(x)b c_2(\mu_0+2)t^{\mu+\alpha}.$$ By the step before it follows that  $b c_{2}(\mu_0+2) =0$ and thus  $c_{2}=0.$
		Successively we see that $c_{i}=0$ for all $i\in \bbN.$

		\noindent  For $\lambda=\mu-\alpha$ we get
		$$\mu(x)c_{-1}t^{\mu - \alpha} = (\mu-\alpha(x))c_{-1}t^{\mu-\alpha}- \alpha(x)b\mu_0 c_0 t^{\mu-\alpha}.$$
		Hence $b\mu_0c_0 = -c_{-1}$. 
		
		\noindent  For $\lambda=\mu-2\alpha$ we get
		$$\mu(x)c_{-2}t^{\mu - 2\alpha}  = (\mu- 2 \alpha(x))c_{-2}t^{\mu-2\alpha}-\alpha(x)b(\mu_0-1)c_{-1}t^{\mu-2\alpha}.$$
		Hence $b(\mu_0-1)c_{-1}=-2 c_{-2}$ and $c_2=\frac{-b(\mu_0-1)c_1}{2}.$ By substituting $c_{-1}$ be the above expression  we get
		$$c_{-2}=\frac{b^2\mu_0(\mu_0-1)c_0}{2}.$$
		By iterating this process we get
		$$c_{-i}=(-1)^i\frac{b^i\mu_0(\mu_0-1) \cdots (\mu_0-(i-1))c_0}{i ! }.$$ 
		Hence we see that the series (\ref{seriesv}) coincides up to sign  with the series in the first case. Hence the argumentation is symmetric  and we get a diverging power series.

		\vskip10pt
		\noindent  {\bf Third Case:} $a\neq 0, b\neq 0$. 
		Set $\bar{\mu}=\mu_0-\mu_1.$ We suppose that $|\mu_0|>|\mu_1|.$
		Thus $|\bar{\mu}|=|\mu_0|.$ The case $|\mu_0|<|\mu_1|$ is treated by symmetry.
		
		{\bf Case A:} $|\mu_0|>1.$ Hence $1> |\mu_1|$ by our assumption.

		\noindent  For $\lambda=\mu$ we get 
		$$\mu(x)c_0t^\mu+ \alpha(x)a (b\bar{\mu}c_0 +  c_{-1}(\mu_1+1)) t^\mu = \mu(x)c_0t^\mu +(b^2a-b)\alpha(x)f c_1t^{\mu + \alpha}.$$
		Hence $a(b\bar{\mu}c_0 + c_{-1}(\mu_1+1)) = (b^2a-b)c_1 (\mu_0+1)$ or 
		$$c_1= \frac{ab\bar{\mu}c_0+a(\mu_1+1)c_{-1}}{(b^2a-b)(\mu_0+1)}.$$
		For $\lambda=\mu+\alpha$ we get
		$$\mu(x)c_1t^{\mu + \alpha}+\alpha(x) a(b(\bar{\mu} + 2)
		c_1 t^{\mu+\alpha}+ e c_0t^\mu)= (\mu+\alpha(x))c_1t^{\mu+\alpha} +(b^2a-b)\alpha(x)(f c_2t^{\mu + 2 \alpha}). $$
		Hence $a(b(\bar{\mu} + 2)c_1 + \mu_1c_0) = c_1+ (b^2a-b)(\mu_0+2)c_2$ and thus
		$$c_2=\frac{c_1(ab(\bar{\mu}+2)-1)+a\mu_1c_0}{(b^2a-b)(\mu_0+2)}. $$
		For $\lambda=\mu+2\alpha$ we get
		$$\mu(x)c_2t^{\mu + 2\alpha}+\alpha(x)a (b(\bar{\mu} + 4)
		c_2 t^{\mu+2\alpha}+ e c_1 t^{\mu+\alpha}) = (\mu+ 2 \alpha(x))c_2t^{\mu+2\alpha}+ (b^2a-b)\alpha(x)(f c_3t^{\mu + 3 \alpha}).$$
		Hence $a(b(\bar{\mu} + 4)c_2 + (\mu_1 - 1)c_1) = 2 c_2 +(b^2a-b)(\mu_0+3)c_3$ and thus
		$$c_3=\frac{c_2(ab(\bar{\mu} + 4)-2) + a(\mu_1 - 1)c_1}{(b^2a-b)(\mu_0+3)}.$$ 
		By iterating this process we get
		$$a(b(\bar{\mu} + 2i)c_i + (\mu_1 - i +1)c_{i-1})= i c_i + (b^2a-b)(\mu_0+i+1)c_{i+1}$$ and thus
		$$c_{i+1}=\frac{c_i(ab(\bar{\mu} + 2i)-i) + a (\mu_1 - i +1)c_{i-1}}{(b^2a-b)(\mu_0+i+1)}$$
		for all $i\in \bbN.$
		
		\noindent  Next we examine the negative part.
		\noindent  For $\lambda=\mu-\alpha$ we get
		$$\mu(x)c_{-1}t^{\mu - \alpha}+\alpha(x)a (b(\bar{\mu} - 2)
		c_{-1} t^{\mu-\alpha}+ e c_{-2}t^{\mu-2\alpha}) = (\mu-\alpha(x))c_{-1}t^{\mu-\alpha} + (b^2a-b)\alpha(x)fc_0 t^{\mu}. $$
		Hence $a(b(\bar{\mu} - 2)c_{-1} + (\mu_1+ 2)c_{-2}) = -c_{-1} + (b^2a-b)\mu_0c_0$. Hence
		$$c_{-2}=\frac{(b^2a-b)\mu_0 c_0-(ab(\bar{\mu}-2)+1)c_{-1}}{a(\mu_1+2)}. $$
		For $\lambda=\mu-2\alpha$ we get
		$$\mu(x)c_{-2}t^{\mu - 2\alpha}+\alpha(x)a(b(\bar{\mu} - 4)
		c_{-2} t^{\mu-2\alpha}+ e c_{-3} t^{\mu-3\alpha}) = (\mu- 2 \alpha(x))c_{-2}t^{\mu- 2\alpha}+  (b^2a-b)\alpha(x)f c_{-1}t^{\mu-\alpha}.$$
		Hence 
		$a(b(\bar{\mu} - 4)c_{-2} + (\mu_1+3)c_{-3}) = -2 c_{-2} + (b^2a-b)c_{-1}(\mu_0-1)$ and so
		$$c_{-3}=\frac{(b^2a-b)c_{-1}(\mu_0-1)-(ab(\bar{\mu}-4)+2)c_{-2}}{a(\mu_1+3)}.$$ 
		By iterating this process we get 
		\begin{numequation}\label{cminusi}
			c_{-i}=\frac{(b^2a-b)c_{-i+2}(\mu_0-i+2)-(ab(\bar{\mu}-2(i-1))+i-1))c_{-i+1}}{a(\mu_1+i)}
		\end{numequation}
		for all $i\in \bbN.$

		If two successive elements $c_k,c_{k+1}$ vanish then one checks that all coefficients $c_i$ vanish, so that $v=0.$

		We make the following case distinction:
		
		\noindent  {\it Case $|ab\mu_0| > 1:$} Here we show that the Laurent series part of the series (\ref{seriesvmitd}) does not converge.
		
		\noindent  Since $|\mu_0| > |ab\mu_0|>1$ it follows that
		$|(ab(\bar{\mu}-2i)+i|=|ab||\mu_0|$ and $t_i:=|\mu_1+i|\leq \max\{|\mu_1|,|i|\}\leq 1$ for all $i\in \bbZ$. Further $|b^2a-b|=|b(ba-1)|=|b|$ since $ba\in m_L.$
		
		If $|ac_{-1}| = |c_0|$ then one checks by the maximum principle
		of the norm $|\cdot|$ that $|c_i|=|a^i c_0|$ for all $i \in \bbN$ since $|ab\mu_0|>1$ and $|\mu_1|\leq 1$ Then we are essentially in the First case ($a \neq 0$, $b=0)$. The difference is that here we even need not to consider the factor $1/i!.$ With the same argument we conclude that the series does not converge. 
		
		If $|ac_{-1}| \neq |c_0|$  it follows that 
		$$|c_{-2}|=\frac{\max\{|b||\mu_0||c_{0}|,  |ab||\mu_0||c_{-1}|\}}{|a|t_{2}}> |b||\mu_0||c_{-1}|.$$
		
		We claim that $|ac_{-2}| \neq |c_{-1}|$. Indeed suppose that $|ac_{-2}| = |c_{-1}|$. By multiplication the above inequality with $|a|$ we get
		$|c_{-1}| = |ac_{-2}| \geq \frac{|ab||\mu_0||c_{-1}|}{t_{2}}.$ But $|ab||\mu_0|>1$ and $t_2\leq 1.$ Hence we get a contradiction.
		Inductively we see that $|ac_{-i}| \neq |c_{-i+1}|$ for all $i\geq 1$ and therefore   that 
		\begin{equation*}
			|c_{-i-1}|=\frac{\max\{|b||\mu_0||c_{-i+1}|,  |ab||\mu_0||c_{-i}|\}}{|a|t_{-i-1}} >|b\mu_0c_{-i}|.
		\end{equation*}
		Then by successive application of this inequality we see that it is enough to show that the series $\sum_{i\leq 0} \frac{b^{-i}c_0}{(p^{m_0+m})^{-i}} (y^{(m)})^{-i}$ does not converge if $u^-\not\in T_{0,r},$ i.e. $b\not\in p^{m_0+m}O_L$.  This is done as in the previous case. 
		
		
		\vspace{0.5cm}
		\noindent The case  $|ab\mu_0| = 1$ is not possible by our assumption that $|\mu_0|\not\in |L|.$ 
		\vspace{1cm}
		
		\noindent  {\it Case $|ab\mu_0| < 1:$} 
		We argue similarly as in the previous case but consider the main series of (\ref{seriesvmitd}). Since this time $d_i=\frac{c_i}{\mu_1(\mu_1-1)\cdots (\mu_{1}-i+1) p^{i(m_0+m)}}$ and $|\mu_1|< 1$ it suffices to see that the series
		$\sum_i \frac{c_i}{p^{i(m_0+m)}} (y^{(m)})^i$ does not converge. Consider
		$$c_{i}=\frac{c_{i-1}(ab(\bar{\mu} + 2(i-1))-(i-1)) + a (\mu_1 - i+2)c_{i-2}}{(b^2a-b)(\mu_0+i)}.$$
		We have $|\bar{\mu} + 2(i-1))|=|\mu_0|$. 
		Suppose that 
		$|c_{i-1}(ab\mu_0-(i-1))| \neq |a (\mu_1 - i+2)c_{i-2}|$. Then we have 
		$$|c_i|\geq  \frac{|c_{i-1}ab\mu_0|}{|b\mu_0|}=|c_{i-1}||a|.$$
		by our assumption $|\mu_0|\not\in |L|$ again. If $p\nmid i-1$ then we even have $|c_i|>|c_{i-1}||a|.$
		
		Consider now $c_{i+1}$. If 
		$|c_{i}(ab\mu_0-i)| \neq |a (\mu_1 - i+1)c_{i-1}|$ then we get as above 
		$|c_{i+1}| \geq  |c_{i}||a|.$
		If 
		$|c_{i}(ab\mu_0-i)| = |a (\mu_1 - i+1)c_{i-1}|$ and $p\nmid i-1$, then we may  replace $c_j$ by $c_j\nu$ for all $j\geq i$ for some $\nu\in \overline{K}$ with $|\nu|<1$ such that $|\nu||c_i|>|c_{i-1}||a|$ since for disproving the convergence of the series (\ref{seriesvmitd}) we can consider the modified series with the "smaller" coefficients, as well. Then we are in the previous case.  If $p\mid i-1$ then $|\mu_1-i+1|<1.$ Since
		$d_i=\frac{c_i}{\mu_1\cdots (\mu_1-i+1)}$ we may replace $c_j$ by $c_j\nu$ for all $j\geq i$ for some $\nu\in \overline{K}$ with $|\mu_1-i+1|<|\nu|<1$ such that $|\nu c_i|/|\mu_1-i+1|>|c_{i-1}||a|.$ Then we are again in the situation considered before.
		
		Thus it suffices to prove that the series 
		$\sum_i \frac{a^i c_0}{p^{i(m_0+m)}} (y^{(m)})^i$ does not converge for $a\not\in p^{m_0+m}O_L$ which is done as in the previous cases.

		{\bf Case B: $|\mu_0| <1 :$}
		
		We argue similarly as before. Consider
		$$c_{i}=\frac{c_{i-1}(ab(\bar{\mu} + 2(i-1))-(i-1)) + a (\mu_1 - i+2)c_{i-2}}{(b^2a-b)(\mu_0+i)}.$$ Suppose that 
		$|c_{i-1}(ab(\bar{\mu} + 2(i-1))-(i-1))| \neq |a (\mu_1 - i+2)c_{i-2}|$. Then we have 
		$$|c_i|\geq  \frac{|c_{i-1}||i-1|}{|b|} > |c_{i-1}(i-1)|.$$
		If $p\nmid i-1$ then we even have $|c_i|>|c_{i-1}|.$
		If $p\mid i-1$ then we consider the inequality
		$$|c_i|\geq  \frac{|c_{i-1}||ab||\mu_0-2(i-1)|}{|b|} > |c_{i-1}||ab||\mu_0-2(i-1)|.$$ Hence since $p\neq 2$ and $|\mu_0| \geq |\mu_1|$
		$$\frac{|c_i|}{|\mu_1-i+1|}>|c_{i-1}||ab|.$$
		Consider now $c_{i+1}$. If 
		$|c_{i}(ab(\bar{\mu} + 2i)-i)| = |a (\mu_1 - i+1)c_{i-1}|$ then we may  replace again as above $c_j$ by $c_j\nu$ for all $j\geq i$ for some $\nu\in \overline{K}$ with $|\mu_1-i+1| < |\nu|<1$ such that still the above inequalities hold and we may assume that $|c_{i}(ab(\bar{\mu} + 2i)-i)| \neq |a (\mu_1 - i+1)c_{i-1}|.$ 
		Then for disproving the convergence  it suffices to prove that the series¸ 
		$\sum_i \frac{(ab)^i c_0}{p^{i(m_0+m)}} (y^{(m)})^i$ does not converge for $a,b\not\in p^{m_0+m}O_L$ which is done as in the previous cases.

	\end{proof}
	
	\vspace{1cm}
	\section{Cuspidality}
	
	We consider  inside $\cW_\kappa$ the subcategory $\cC_\kappa:=\cC \cap \cW_\kappa$ of cuspidal modules. 
	Let $M\in \cC_\kappa$, $Z\in \Rep^{alg}_K(P_{(1,n)}), V \in \Rep^\infty_K(T_0)$ and recall that $X=X({M,Z,V}):=  D(G_0) \otimes_{D(\frg,T_0)} M(Z)\otimes V'.$ In this final section we are going to prove that the representation $\cF^G_\kappa(M,Z,V)=c-\Ind^G_{G_0} X'$ satisfies the vanishing of the homology groups in the definition of supercuspidality in the sense of Kohlhaase. We start with the following result.

	\begin{prop}
		Let $P\subset G$ be a parabolic subgroup and $N=U_P$ its unipotent radical. Set $N_0=G_0 \cap N.$ Then $H^0(N_0,X)=0.$
	\end{prop}

	\begin{proof}
		We follow the proof of  \cite[Thm. 3.5]{OSch}. We clearly have $H^0(N_0,X)=H^0(\frn,X)^{N_0}$ so that it suffices to see that  $H^0(\frn,X)=0.$ Since the action of $\frn$ on $V$ is trivial we may suppose that $V={\bf 1}^\infty$ is the trivial representation, cf \cite[Lemma 3.4]{OSch}.  

		We write $X$ in the shape 
		$$X=	\varprojlim_r \cM(Z)_r$$    where $\cM(Z)_r=D_r(G_0) \otimes_{D(\frg, T_0)} M(Z)$. If we denote
		by $M(Z)_r$ the topological closure\footnote{here we use the notation of \cite{OSch}.} of $M$ in $\cM(Z)_r$, we get by \cite[5.6.5]{OS2} finitely many elements  $g \in G_{0}$ such that
		\begin{equation}\label{equation_split} \cM(Z)_r \simeq \bigoplus_{g} \delta_g \otimes M(Z)_r
		\end{equation}
		and the action of $\frx \in \frn$ is given by
		\begin{equation*}
			\frx \cdot \sum_{g} \delta_g \otimes m_{g}=\sum \delta_g \otimes
			\mathrm{Ad}(g^{-1})(\frx)m_{g}.
		\end{equation*}

		But all elements of $\mathrm{Ad}(g^{-1})(\frn)$ act injectively on $M(Z)$ since $M(Z)$ is cuspidal. By a similar argument as in  Step 1 of \cite[Theorem 5.7]{OS2} they act 
		injectively on
		$M(Z)_r$, as well since each element in $M(Z)_r$ is again a power series in weight vectors. We conclude that $H^0(\mathrm{Ad}(g^{-1})(\frn), M(Z)_r)=0.$ 	Hence by passing to the limit we
		get $H^0(\frn, \cM(Z))=0$. The claim follows.  
	\end{proof}

	\begin{lemma}\label{lemma_action_biejctive}
		With the notation above the Lie algebra $\frn$ acts bijectively on $X.$
	\end{lemma}
	
	\begin{proof}
		Again we may suppose that $V={\bf 1}^\infty$ is the trivial representation.
		In the above proposition we have actually seen that any element $s\in \frn$ acts injectively on $X.$
		Now we show that it acts surjectively, as well. We have thus to show that $s\cdot D(G_0) \otimes_{D(\frg,T_0)} M(Z)= D(G_0) \otimes_{D(\frg,T_0)} M(Z).$ For this it is enough to see that the tensors $\delta \otimes m$ with $\delta \in D(G_0)$ and $m\in M(Z)$ are in the image of the multiplication map by $s.$. 
		Now by \cite[Lemma 4.2]{Ma} the  ad-nilpotent elements $\frn$ satisfy  Ore's  localizability condition in $D(G_0)$.
		By definition there are $\delta' \in D(G_0)$ and $s'\in S$ with
		$s\delta'=\delta s'.$ Since $s'$ acts bijectively on $M(Z)$ as $M(Z)$ is cuspidal there is some $m'\in M(Z)$ with $s'm'=m$. Hence $\delta \otimes m$=$\delta \otimes s'm'=\delta s'\otimes m' =s\delta'  \otimes m'$ and the claim follows.
	\end{proof}
	
	\begin{prop}
		We have $H_i(N,c-\Ind^G_{G_0}(X'))=0$ for all $i \geq 0.$
	\end{prop}
	
	\begin{proof}
		By a result of Kohlhaase \cite[Theorem 7.1]{K2} we have   $H_i(N,c-\Ind^G_{G_0}(X'))=H_i(\frn, c-\Ind^G_{G_0}(X'))$ for all $i\geq 0.$ 
		Since $c-\Ind^G_{G_0}(X')$ is a direct sum $\bigoplus_{g\in G/G_0} gX'$ of copies of $X'$ and $H_i(\frn,gX')=H_i(Ad(g)(\frn),X')$  it suffices to see that $H_i(\frn, X')=H_i(U(\frn), X')=0$ since we consider from the very beginning arbitrary parabolic subgroups $P.$
		
		By Lemma \ref{lemma_action_biejctive} elements of $\frn$ act bijectively on $X$. Thus they act bijectively on $X'$, as well, and we can consider $X'$ as a module for the localisation $S^{-1}U(\frg)$ where $S=\frn U(\frn) \subset U(\frg)$ is the multiplicative system generated by $\frn.$  
		But now projective modules remain projective under the localisation functor $S^{-1} -$ which is  moreover exact.
		It follows that any projective resolution of $X'$ by $U(\frg)$-modules gives rise to a  projective resolution of $X'$ by $S^{-1}U(\frg)$-modules.  
		But for any $S^{-1}U(\frg)$-module $Y$ we have $H_0(\frn,Y)=0$
		since elements of $\frn$ act bijectively on $Y.$ Hence all cohomology groups vanish.  
	\end{proof}

	\newpage
	
	\appendix
	\section{Hecke operators on ind-admissible representations, by Andreas Bode}
	
	The purpose of this appendix is twofold: Firstly, we provide some additional background on the category of ind-admissible representations, advertising a dual perspective using pro-coadmissible modules over the distribution algebra -- in particular, we show that the definition of ind-admissibility is independent of the choice of compact open subgroup (Corollary \ref{indep}), and that ind-admissible representations form an abelian category (Proposition \ref{iaabelian}). Secondly, we describe more explicitly the $\mathrm{GL}_2(L)$-representation obtained from the cuspidal Lie algebra representation $M^\mu$, and produce non-trivial Hecke operators on it. In particular, the representation
	\begin{equation*}
		W:=\mathrm{c-Ind}_{G_0Z_G}^G((D(G_0)\otimes_{D(\mathfrak{g}, T_0)}M^\mu)')
	\end{equation*}
	is not topologically irreducible, and one is led to consider the representations
	\begin{equation*}
		W\otimes_{\mathcal{H}, \rho}K
	\end{equation*} 
	for a fixed Hecke character $\rho: \mathcal{H}=\mathrm{End}_G(W)^{\mathrm{op}}\to K$.
	
	\subsection{Ind-admissible representations and pro-coadmissible modules}
	Throughout, let $L$ be a finite extension of $\mathbb{Q}_p$ and let $K$ be a spherically complete nonarchimedean field extension of $L$. Let $H$ be a locally $L$-analytic group, and let $C\leq H$ be a compact open subgroup.
	
	We write $D(H)$ for the distribution algebra $D(H, K)$ of locally $L$-analytic distributions, likewise $D(C)=D(C, K)$. Choose any increasing sequence $1/p\leq r_n<1$ tending to $1$. Then \cite{ST2} defines Noetherian Banach algebras $D_{r_n}(C, K)$ such that $D(C)\cong \varprojlim D_{r_n}(C, K)$ exhibits $D(C)$ as a Fr\'echet--Stein algebra. We will shorten $D_{r_n}(C, K)$ to $D_n(C)$.
	
	Recall from Definition \ref{dfn_qa} that a locally $L$-analytic $H$-representation $V$ is called ind-admissible if it can be written as the inductive limit $V\cong \varinjlim V_i$ of a countable chain of admissible $C$-subrepresentations.
	
	Note that such an inductive limit is automatically a strict inductive limit, as each morphism of admissible $C$-representations is strict (\cite[Proposition 6.4]{ST2}).
	
	For now, we have fixed $C$ and will say `ind-admissible' when we really mean `ind-admissible relative to $C$'. We will see at the end of this section that our definition does not depend on the choice of compact open subgroup.
	
	Note that by Lemma \ref{iacompact}, any ind-admissible representation $V$ is of compact type -- in particular, $V$ is reflexive and its dual $V'$ is a nuclear Fr\'echet space.
	
	Explicitly, $V|_C\cong \varinjlim V_i$ and $V_i\cong \varinjlim V_{i, n}$, where the $V_{i, n}=(D_n(C)\otimes_{D(C)} V'_i)'$ are the Banach spaces consisting of vectors in $V_i$ with a prescribed radius of analyticity. In particular, there are natural morphisms $V_{i, n}\to V_{i+1, n}$ such that the composition $V_{i, i}\to V_{i+1, i}\to V_{i+1, i+1}$ is compact by \cite[Remark 16.7]{S1}, making $V\cong \varinjlim V_{i, i}$ a space of compact type.
	
	
	
	

	We note that the maps $V_i\to V_{i+1}$ are generally not compact. We thus have two different presentations for $V$ with distinct advantages and disadvantages: $V\cong \varinjlim V_i$ is a strict inductive limit with transition maps which are not necessarily compact, and $V\cong \varinjlim V_{i, i}$ realizes $V$ as a (generally non-strict) inductive limit of Banach spaces with compact transition maps.
	
	\begin{lemma}
		\label{limiso}
		Let $V$ be an ind-admissible $H$-representation. As before, write $V\cong \varinjlim V_i$ and $V\cong \varinjlim V_{i, i}$. The natural morphisms
		\begin{equation*}
			V'\to \varprojlim V'_i
		\end{equation*}
		and
		\begin{equation*}
			V'\to \varprojlim V'_{i, i}
		\end{equation*}
		are isomorphisms of Fr\'echet spaces, where the right hand side is equipped with the inverse limit topology.
	\end{lemma}
	\begin{proof}
		Since the $V_{i, i}$ exhibit $V$ as a space of compact type, the second isomorphism follows from \cite[Proposition 16.10]{S1}.
		
		But then $V'\cong \varprojlim_{i, j} V'_{i, j}$, and the first isomorphism follows from $V'_i\cong \varprojlim_j V'_{i, j}$ for each $i$.
	\end{proof}
	
	\begin{cor}
		\label{closed}
		Let $V$ be an ind-admissible $H$-representation and let $U$ be a closed subrepresentation. Then $U$ and $V/U$ are also ind-admissible.
	\end{cor}
	\begin{proof}
		Let $V|_C=\varinjlim V_i$, with each $V_i$ an admissible $C$-representation. Set $U_i=V_i\cap U$, a closed, and hence admissible, $C$-subrepresentation of $V_i$.
		
		Write $W=V/U$ and let $W_i=V_i/U_i$, an admissible $C$-representation. By construction, the natural map $W_i\to W_{i+1}$ is injective. Since $\varinjlim$ is right exact on the category of locally convex $K$-vector spaces, $W=V/U\cong \varinjlim V_i/U_i= \varinjlim W_i$, so $W$ is an ind-admissible representation.
		
		It remains to verify that $U$ is actually isomorphic to $\varinjlim U_i$, i.e. that the subspace topology on $U$ agrees with the inductive limit topology. 
		
		For this, note that $U$, $V$, $W$ are spaces of compact type by \cite[Proposition 1.2]{ST1}, and \cite[Proposition 1.2]{ST1} then induces a short strictly exact sequence
		\begin{equation*}
			0\to W'\to V'\to U'\to 0,
		\end{equation*}
		of nuclear Fr\'echet spaces, fitting into a commutative diagram
		\begin{equation*}
			\begin{xy}
				\xymatrix{0\ar[r]&W'\ar[r]\ar[d]&V'\ar[r]\ar[d]& U'\ar[r]\ar[d]&0\\
					0\ar[r]& \varprojlim W_i'\ar[r]& \varprojlim V'_i\ar[r]& \varprojlim U'_i.}
			\end{xy}
		\end{equation*}
		By Lemma \ref{limiso}, the first two vertical arrows are isomorphisms, and by reflexivity, it suffices to show that the third vertical arrow is also an isomorphism.
		
		Each transition map $W'_i\to W'_{i-1}$ is a surjection of Fr\'echet spaces, so $(W'_i)$ is $\varprojlim$-acyclic by \cite[p.45, Lemme 1]{Douady}. Hence $\varprojlim V'_i\to \varprojlim U'_i$ is surjective, and thus $U'\to\varprojlim U'_i$ is a continuous bijection of Fr\'echet spaces and therefore an isomorphism by the Open Mapping Theorem (\cite[Corollary 8.7]{S1}). 
	\end{proof}
	
	The above readily yields the following.
	
	\begin{cor}
		\label{procoad}
		Let $V$ be a locally analytic $H$-representation whose underlying locally convex topological $K$-vector space is of compact type. Then $V$ is ind-admissible if and only if $V'$ is a (separately continuous, nuclear Fr\'echet) $D(H)$-module which can be written as an inverse limit $V'\cong \varprojlim M_i$ with each $M_i$ a coadmissible $D(C)$-module with surjective transition maps $M_{i+1}\to M_i$.
	\end{cor}
		
		
	
	Accordingly, we make the following definition.
	
	\begin{dfn}
		A nuclear Fr\'echet $D(C)$-module $M$ is called \textbf{pro-coadmissible} if $M\cong \varprojlim M_i$, where each $M_i$ is a coadmissible $D(C)$-module and each transition map $M_{i+1}\to M_i$ is surjective.
		
		A separately continuous, nuclear Fr\'echet $D(H)$-module is called pro-coadmissible if it is pro-coadmissible as a $D(C)$-module.
	\end{dfn}
	
	We will prove later that the notion of a pro-coadmissible $D(H)$-module does not depend on the choice of compact open subgroup $C$.
	
	Corollary \ref{procoad} then says that the duality in \cite[Corollary 3.3]{ST1} induces an anti-equivalence between ind-admissible $H$-representations and pro-coadmissible $D(H)$-modules.
	
	Note that by Lemma \ref{limiso} and \cite[Proposition 16.10]{S1}, if $M$ is a pro-coadmissible $D(C)$-module, then
	\begin{equation*}
		M\cong \varprojlim_i (D_i(C)\h{\otimes}_{D(C)} M_i).
	\end{equation*}
	
	It is occasionally advantageous to use the $D(C)$-module perspective, as we can regard pro-coadmissible modules (which are by definition continuous $D(C)$-modules on nuclear Fr\'echet spaces) as special cases of complete bornological $D(C)$-modules. For example, this allows us to exploit a well-behaved theory of completed tensor products: The category $\h{\B}c_K$ of complete bornological $K$-vector spaces is closed symmetric monoidal, unlike e.g. the category $LCVS_K$ of locally convex topological $K$-vector spaces.
	
	We refer to \cite{SixOp} for preliminaries on the category $\mathrm{Mod}_{\h{\B}c_K}(D(C))$ of complete bornological $D(C)$-modules.
	
	Recall from \cite[subsection 4.1]{SixOp} that there exists a bornologification functor 
	\begin{equation*}
		(-)^b: LCVS_K\to \B c_K
	\end{equation*}
	from locally convex topological $K$-vector spaces to bornological vector spaces.
	
	When restricted to nuclear Fr\'echet spaces, $(-)^b$ is exact and fully faithful (see \cite[Lemma 4.3]{SixOp} for full faithfulness and e.g. \cite[Proposition 5.12]{SixOp} for exactness), with a quasi-inverse on the essential image given by the functor $(-)^t$.
	
	Let $\mathrm{Mod}_{\h{\B}c_K}^{\mathrm{nF}}(D(C))$ denote the full subcategory of $\mathrm{Mod}_{\h{\B}c_K}(D(C))$ whose underlying complete bornological vector space is of the form $E^b$ for some nuclear Fr\'echet space $E$. 
	
	\begin{lemma}
		\label{bembedding}
		The bornologification functor $(-)^b$ induces a fully faithful functor
		
		\begin{equation*}
			(-)^b: 
			\left\{\begin{matrix}
				\text{continuous $D(C)$-modules}\\
				\text{on nuclear Fr\'echet spaces}\\
				\text{with continuous $D(C)$-module morphisms}
			\end{matrix}\right\}\to \mathrm{Mod}_{\h{\B}c_K}(D(C))
		\end{equation*}
		with essential image $\mathrm{Mod}_{\h{\B}c_K}^{\mathrm{nF}}(D(C))$.
		
		If $M$ is nuclear Fr\'echet $D(C)$-module and $N$ is a closed submodule (so that $N$ and $M/N$ are nuclear Fr\'echet by \cite[Proposition 19.4]{S1}), then
		\begin{equation*}
			0\to N^b\to M^b\to (M/N)^b\to 0
		\end{equation*}
		is strictly exact.
	\end{lemma} 
	\begin{proof}
		By \cite[Theorem 8.5.7, Theorem 10.3.13, Theorem 10.4.4]{Schikhof} and \cite[Corollary 5.20]{SixOp}, if $M$ is a nuclear Fr\'echet $K$-vector space, then $D(C)\h{\otimes}_K M$ is a nuclear Fr\'echet space and
		\begin{equation*}
			(D(C)\h{\otimes}_K M)^b\cong D(C)^b\h{\otimes}_K M^b.
		\end{equation*}
		Thus $(-)^b$ induces a functor between the module categories, which inherits exactness and fully faithfulness from $(-)^b$ (since morphisms in the module category are morphisms in $\h{\B}c_K$ fitting into suitable commutative diagrams, using the equality of tensor products above).
		
		Moreover, 
		\begin{equation*}
			((D(C)\h{\otimes}_K M)^b)^t\cong D(C)\h{\otimes}_K M
		\end{equation*}
		by \cite[Lemma 4.3]{SixOp}, so specifying a continuous $D(C)$-module structure on $M$ is equivalent to specifying a bounded action on $M^b$. Hence $(-)^b$ has essential image $\mathrm{Mod}_{\h{\B}c_K}^{\mathrm{nF}}(D(C))$. 
	\end{proof}
	
	We usually suppress the $(-)^b$ from the notation and from now on regard pro-coadmissible modules as certain complete bornological modules.
	
	\begin{lemma}
		\label{changeton}
		Let $M\cong \varprojlim M_i$ be a pro-coadmissible $D(C)$-module. Then the natural morphism
		\begin{equation*}
			D_n(C)\h{\otimes}_{D(C)} M\to \varprojlim_i (D_n(C)\h{\otimes}_{D(C)} M_i)
		\end{equation*}
		is an isomorphism in $\mathrm{Mod}_{\h{\B}c_K}(D(C))$.
		
		In particular, the morphism $M\to \varprojlim_n (D_n(C)\h{\otimes}_{D(C)}M)$ is an isomorphism.
	\end{lemma}
	\begin{proof}
		By \cite[Lemma 5.31]{SixOp}, $D_n(C)\h{\otimes}^{\mathbb{L}}_{D(C)}M \in \mathrm{D}(\mathrm{Mod}_{\h{\B}c_K}(D_n(C)))$ can be computed via the complex
		\begin{equation*}
			\hdots\to D_n(C)\h{\otimes}_K D(C)^{\h{\otimes}j}\h{\otimes}_K M\to D_n(C)\h{\otimes}_K D(C)^{\h{\otimes}j-1}\h{\otimes}_K M\to \hdots \to D_n(C)\h{\otimes}_K M\to 0
		\end{equation*} 
		obtained from the Bar resolution of $M$.
		
		Since each $M_i$ is a coadmissible $D(C)$-module, $D_n(C)\h{\otimes}_{D(C)}M_i$ is simply the bornologification of the finitely generated $D_n(C)$-module $M_{i, n}=D_n(C)\otimes_{D(C)} M_i$ from before (see \cite[Corollary 5.38]{SixOp}).
		
		It also follows from \cite[Lemma 5.32]{SixOp} that the complex
		\begin{equation*}
			\hdots \to D_n(C)\h{\otimes}_K D_i(C)^{\h{\otimes}j}\h{\otimes}_K M_{i, i}\to \hdots \to D_n(C)\h{\otimes}_{D_i(C)} M_{i, i}\to 0 
		\end{equation*}
		is strictly exact.
		
		For each $i$, the inverse system $(D_n(C)\h{\otimes} D_i(C)^{\h{\otimes} j} \h{\otimes} M_{i, i})_i$ is a pre-nuclear system of Banach spaces with dense images as in \cite[Definition 5.24]{SixOp}, since the maps $D_i(C)\to D_{i-1}(C)$ and $M_{i, i}\to M_{i-1, i-1}$ are compact for each $i$.
		
		Hence \cite[Corollary 5.28]{SixOp} implies that
		\begin{equation*}
			\hdots \to\varprojlim_i (D_n(C)\h{\otimes} D_i(C)^{\h{\otimes}j}\h{\otimes} M_{i, i})\to \hdots \to \varprojlim_i (M_{i, n})\to 0
		\end{equation*}
		is strictly exact.
		
		But by \cite[Corollary 5.22]{SixOp}, we now have
		\begin{equation*}
			\varprojlim_i (D_n(C)\h{\otimes}_K D_i(C)^{\h{\otimes}j}\h{\otimes}_K M_{i, i})\cong D_n(C)\h{\otimes}_K D(C)^{\h{\otimes}j} \h{\otimes}_KM,
		\end{equation*}
		so the result follows.
		
		Taking the limit over $n$ and applying Lemma \ref{limiso} yields the final isomorphism
		\begin{equation*}
			M\cong \varprojlim_n (D_n(C)\h{\otimes}_{D(C)} M).
		\end{equation*}
	\end{proof}
	
	The lemma above allows us to deduce several useful corollaries.
	
	\begin{cor}
		\label{Dnflat}
		Let
		\begin{equation*}
			0\to M\to M'\to M''\to 0
		\end{equation*}
		be a short strictly exact sequence of pro-coadmissible $D(C)$-modules. Then
		\begin{equation*}
			0\to D_n(C)\h{\otimes}_{D(C)} M\to D_n(C)\h{\otimes}_{D(C)} M'\to D_n(C)\h{\otimes}_{D(C)} M''\to 0
		\end{equation*}
		is short strictly exact.
	\end{cor}
	\begin{proof}
		We have shown above that $D_n(C)\h{\otimes}_{D(C)}^{\mathbb{L}}M\cong D_n(C)\h{\otimes}_{D(C)} M$ for any pro-coadmissible $D(C)$-module $M$.
	\end{proof}
	
	\begin{cor}
		\label{coadcheck}
		A pro-coadmissible module $M$ is coadmissible if and only if the $D_n(C)$-module $D_n(C)\h{\otimes}_{D(C)}M$ is finitely generated and equipped with its canonical Banach structure for each $n$.
	\end{cor}
	
	\begin{cor}
		\label{coadinj}
		Let $N$ be a coadmissible $D(C)$-module and $M$ a pro-coadmissible $D(C)$-module. If there exists an injective morphism $j: M\to N$, then $M$ is coadmissible and $j$ is strict.
	\end{cor}
	
	\begin{proof}
		
		It suffices to show that $M$ is coadmissible, the strictness of the morphism then follows from \cite[Corollary 5.14]{SixOp}. 
		
		Since $D_n(C)$ is flat as an abstract $D(C)$-module, $j$ induces an injection
		\begin{equation*}
			D_n(C)\otimes_{D(C)} M\to D_n(C)\otimes_{D(C)} N.
		\end{equation*}
		But the right-hand side is a finitely generated $D_n(C)$-module, so $D_n(C)\otimes_{D(C)}M$ is also finitely generated and the morphism is strict when both sides are equipped with their canonical Banach structure.
		
		But now the natural bounded bijection
		\begin{equation*}
			(D_n(C)\otimes M, \tau_1)\to (D_n(C)\otimes M, \tau_2),
		\end{equation*}
		where $\tau_1$ denotes the tensor product bornology and $\tau_2$ the subspace bornology from $D_n(C)\otimes N$ (i.e. the canonical Banach structure), has a bounded inverse by \cite[Lemma 4.27.(ii)]{SixOp}, so it is an isomorphism. In particular, $D_n(C)\otimes_{D(C)} M$ is already complete, and
		\begin{equation*}
			D_n(C)\h{\otimes}_{D(C)} M\cong D_n(C)\otimes_{D(C)} M
		\end{equation*}
		is indeed a finitely generated $D_n(C)$-module, endowed with its canonical Banach structure. 
		
		We can thus apply Corollary \ref{coadcheck}.
	\end{proof}
	
	\begin{cor}
		Let $N$ be a coadmissble $D(C)$-module and $M$ a pro-coadmissible $D(C)$-module. Any morphism $M\to N$ is strict.
	\end{cor}
	\begin{proof}
		Let $f: M\to N$ be a morphism in $\mathrm{Mod}_{\h{\B}c_K}(D(C))$. Since $N$ is Fr\'echet, the kernel of $f$ is closed. Hence $\mathrm{coim}f= M/\mathrm{ker}f$ is pro-coadmissible and injects into $N$.
		
		By the above, $\mathrm{coim}f$ is coadmissible and the inclusion $\mathrm{coim}f\to N$ is strict. In particular, $f$ is strict.
	\end{proof}

	\begin{prop}
		\label{levelwise}
		Let $f: M\to N$ be a morphism of pro-coadmissible $D(C)$-modules. Then one can choose presentations $M\cong \varprojlim M_i$, $N\cong \varprojlim N_i$ such that $f$ is induced by a family of morphisms $M_i\to N_i$.
	\end{prop}
	
	\begin{proof}
		Let $M\cong \varprojlim M_i$, $N\cong \varprojlim N_i$. Let $T_i(M)$ denote the kernel of the projection $M\to M_i$. Note that $T_i(M)$ is a closed submodule of $M$ and hence pro-coadmissible.

		Let $g_i: M\to N_i$ denote the morphism obtained by composing $f$ with the natural projection $N\to N_i$, and let $S_i=\mathrm{ker}g_i$. Once again, this is a closed submodule of $M$ and hence pro-coadmissible. 
		
		In particular, $T_i(M)\cap S_i$ is a closed submodule of $M$ and we can consider the pro-coadmissible module $\widetilde{M}_i=M/(T_i(M)\cap S_i)$.

		We now claim the following:		
		\begin{enumerate}
			\item[(i)] Each $\widetilde{M}_i$ is coadmissible.
			
			\item[(ii)] The natural morphisms $M\to \widetilde{M}_i$ induce an isomorphism 
			$M\cong \varprojlim \widetilde{M}_i$.
			
			\item[(iii)] $f$ is obtained as the inverse limit of the morphisms $\widetilde{M}_i\to N_i$.
		\end{enumerate}
		
		For (i), we consider the strictly exact sequence
		\begin{equation*}
			0\to \frac{S_i}{T_i(M)\cap S_i}\to \widetilde{M}_i\to M/S_i\to 0
		\end{equation*}
		of pro-coadmissible $D(C)$-modules. Since $\frac{S_i}{T_i(M)\cap S_i}$ injects into $M_i\cong M/T_i(M)$, it is coadmissible by Corollary \ref{coadinj}. Likewise, $M/S_i$ injects into $N_i$ and is therefore coadmissible.
		
		It follows that $\widetilde{M}_i$ is also coadmissible: Corollary \ref{Dnflat} yields strictly exact sequences
		\begin{equation*}
			0\to D_n(C)\h{\otimes}\frac{S_i}{T_i(M)\cap S_i}\to D_n(C)\h{\otimes} \widetilde{M}_i\to D_n(C)\h{\otimes} M/S\to 0,
		\end{equation*}
		where $D_n(C)\h{\otimes} \frac{S_i}{T_i(M)\cap S_i}$ and $D_n(C)\h{\otimes} M/S$ are finitely generated $D_n(C)$-modules with their canonical Banach structure. Hence, the same is true of the middle term $D_n(C)\h{\otimes} \widetilde{M}_i$, so $\widetilde{M}_i$ is coadmissible by Corollary \ref{coadcheck}.
		
		For (ii), note that the natural surjections $\widetilde{M}_i=M/(T_i(M)\cap S_i)\to M_i=M/T_i(M)$ produce a section to the natural morphism $M\to \varprojlim \widetilde{M}_i$. Thus $M$ is a direct summand of $\varprojlim \widetilde{M}_i$. But since $M$ surjects onto each $\widetilde{M}_i$, it is also a dense subspace, and hence $M\cong \varprojlim \widetilde{M}_i$.
		
		For (iii), it is clear from the definition that $g_i$ factors through $\widetilde{M}_i$, i.e. $g_i$ induces a morphism $f_i: \widetilde{M}_i\to N_i$, and $f=\varprojlim f_i$ by construction.
	\end{proof}
	
	\begin{prop}
		Any morphism of pro-coadmissible $D(C)$-modules is strict.
	\end{prop}
	\begin{proof}
		Let $f: M\to N$ be a morphism of pro-coadmissible $D(C)$-modules. Replacing $M$ by $\mathrm{Coim}f$ and $N$ by $\mathrm{Im}f$, we can assume that $f$ is injective with dense image. We wish to show that in this case, $f$ is an isomorphism.
		
		We write $M\cong \varprojlim M_i$, $N\cong \varprojlim N_i$. By Proposition \ref{levelwise}, we can assume that $f$ is induced by morphisms $f_i: M_i\to N_i$. Note that, since morphisms of coadmissible modules are strict, each $f_i$ is surjective (but not necessarily injective).
		
		In particular, there is a natural isomorphism $M/(M\cap T_i(N))\cong N_i$. 
		
		Let $S_i$ denote the kernel of the natural map $M\to N\to N_{i, i}$, and set $\widetilde{M}_i=M/(S_i\cap T_i(M))$. Since $S_i\cap T_i(M)$ is closed, this is a pro-coadmissible module.
		
		We now claim the following, similar to Proposition \ref{levelwise}:
		\begin{enumerate}
			\item[(i)] Each $\widetilde{M}_i$ is coadmissible. In particular, $\varprojlim \widetilde{M}_i$ is pro-coadmissible and
			\begin{equation*}
				\varprojlim \widetilde{M}_i\cong \varprojlim D_i(C)\h{\otimes}_{D(C)} \widetilde{M}_i
			\end{equation*}
			by Lemma \ref{limiso}.
			\item[(ii)] The natural morphisms $M\to \widetilde{M}_i$ induce an isomorphism $M\cong \varprojlim \widetilde{M}_i$.
			\item[(iii)] For each $i$, $f$ induces an isomorphism 
			\begin{equation*}
				D_i(C)\h{\otimes}_{D(C)} \widetilde{M}_i\to N_{i, i}.
			\end{equation*}
			Taking the limit, it follows that $f$ is an isomorphism.
		\end{enumerate}
		
		For (i), we have a strictly exact sequence
		\begin{equation*}
			0\to \frac{T_i(M)}{T_i(M)\cap S_i}\to \widetilde{M}_i\to M_i\to 0.
		\end{equation*}
		Since $S_i$ contains $T_i(N)\cap M$, $M/S_i$ is a quotient of $N_i$ by a closed submodule and thus coadmissible by \cite[Lemma 3.6]{ST2}. Hence $\frac{T_i(M)}{T_i(M)\cap S_i}\subseteq M/S_i$ is also coadmissible by Corollary \ref{coadinj}. Since $M_i$ is also coadmissible, it follows as before that $\widetilde{M}_i$ is coadmissible.
		
		The proof for (ii) is the same as in Proposition \ref{levelwise}.
		
		For (iii), $f$ induces by construction a morphism
		\begin{equation*}
			D_i\h{\otimes}_{D(C)} \widetilde{M}_i\to N_{i, i}.
		\end{equation*}
		Note that applying Corollary \ref{Dnflat} to the strictly exact sequences
		\begin{equation*}
			0\to M\cap T_i(N)\to M\to N_i\to 0
		\end{equation*}
		and
		\begin{equation*}
			0\to S_i\to M\to \widetilde{M}_i\to 0
		\end{equation*}
		shows that $D_i(C)\h{\otimes}\widetilde{M}_i$ is a quotient of $D_i(C)\h{\otimes} M$, but the kernel of
		\begin{equation*}
			D_i(C)\h{\otimes} M\to D_i(C)\h{\otimes} N_i=N_{i, i}
		\end{equation*}
		consists of $D_i(C)\h{\otimes} (M\cap T_i(N))$, which maps to zero in $D_i(C)\h{\otimes} \widetilde{M}_i$, since $S_i$ contains $M\cap T_i(N)$. Hence the morphism
		\begin{equation*}
			D_i(C)\h{\otimes}\widetilde{M}_i\to N_{i, i}
		\end{equation*}
		is injective.
		
		But since $M\to N_{i, i}$ has dense image and morphisms between finitely generated $D_i(C)$-modules are strict, this yields the desired isomorphism. 
	\end{proof}

	\begin{cor}
		\label{iastrict}
		Any morphism of ind-admissible $H$-representations is strict.
	\end{cor}
	\begin{proof}
		Compare \cite[Proposition 6.4]{ST2}.
		
		If $f: V\to W$ is a morphism of ind-admissible $H$-representations, then the dual map $f': W'\to V'$ is a morphism of pro-coadmissible $D(C)$-modules and hence strict, i.e. we have strictly exact sequences
		\begin{equation*}
			0\to \mathrm{ker}(f')\to W'\to \mathrm{coim}(f')\to 0
		\end{equation*}
		and
		\begin{equation*}
			0\to \mathrm{im}(f')\to V'\to \mathrm{coker}(f')\to 0
		\end{equation*}
		of pro-coadmissible modules, with $f'$ inducing an isomorphism $\mathrm{coim}(f')\cong \mathrm{im}(f')$. 
		
		But then reflexivity implies that $V\to (\mathrm{im}(f'))'$ is a strict epimorphism such that $f$ factors through a monomorphism $(\mathrm{im}(f'))'\to W$. Hence $(\mathrm{im}(f'))'$ is the coimage of $f$ (as the universal morphism $\mathrm{coim}f\to (\mathrm{im}(f'))'$ is then both a strict epimorphism and a monomorphism). By an analogous argument, $(\mathrm{coim}(f'))'$ is the image of $f$, and the isomorphism between the two shows that $f$ is strict.
	\end{proof}

	\begin{cor}
		\label{indep}
		Let $C, C'\leq H$ denote two compact open subgroups of $H$. Then a locally analytic representation $V$ is ind-admissible relative to $C$ if and only if it is ind-admissible relative to $C'$.
	\end{cor}
	\begin{proof}
		By considering $C\cap C'$, it suffices to treat the case where $C'\leq C$.
		
		Let $V$ be an ind-admissible $H$-representation relative to $C$. Then $V|_C\cong \varinjlim V_i$, where each $V_i$ is an admissible $C$-representation. But then $V_i$ is also an admissible $C'$-representation so $V$ is also ind-admissible relative to $C'$ by \cite[Lemma 3.8]{ST2}.
		
		Conversely, if $V$ is an admissible $H$-representation relative to $C'$, write $V|_{C'}\cong \varinjlim V_i$, where each $V_i$ is an admissible $C'$-representation. For each $i$, let $W_i$ denote the $C$-subrepresentation of $V$ generated by $V_i$. 
		
		Then $W_i$ is the image of the natural morphism $\alpha: \mathrm{Ind}_{C'}^CV_i\to V$ induced by the inclusion $V_i\to V$. Since $C'$ is of finite index in $C$, $\mathrm{Ind}_{C'}^CV_i$ is an admissible $C$-representation, and hence also an admissible $C'$-representation, again by \cite[Lemma 3.8]{ST2}. 
		
		In particular, $\alpha$ is a morphism between ind-admissible $C'$-representations, so is strict by the above. Hence $W_i\subseteq V$ is endowed with the quotient topology from $\mathrm{Ind}V_i$ and is thus an admissible $C$-representation by Corollary \ref{closed}. In particular, $\varinjlim W_i$ is an ind-admissible $C$-representation.
		
		As a vector space, $\varinjlim W_i=\cup W_i=V$, so it now remains to show the topologies agree, i.e. that the natural continuous bijection $\varinjlim V_i\to \varinjlim W_i$ is an isomorphism.
		
		But as this is in particular a continuous map of ind-admissible $H$-representations relative to $C'$ (by the easy direction above), it is strict by Corollary \ref{iastrict}, and thus an isomorphism, as required. 
	\end{proof}
	
	The strictness result Corollary \ref{iastrict} also yields an alternative, module-theoretic proof of Proposition \ref{abeliancat}: Writing $\mathrm{Rep}^{\mathrm{ind-adm}}(H)$ to denote the full subcategory of ind-admissible $H$-representations, we obtain the following.
	
	\begin{prop}
		\label{iaabelian}
		\leavevmode
		\begin{enumerate}
			\item[(i)] $\mathrm{Rep}^{\mathrm{ind-adm}}(H)$ is an abelian category.
			\item[(ii)] $\mathrm{Rep}^{\mathrm{ind-adm}}(H)$ is closed under the passage to closed subrepresentations and under countable direct sums (and hence under countable inductive limits by (i)).
		\end{enumerate}
	\end{prop}
	\begin{proof}
		\begin{enumerate}
			\item[(i)] Clearly, $\mathrm{Rep}^{\mathrm{ind-adm}}(H)$ is an additive category. If $f: V\to W$ is a morphism in $\mathrm{Rep}(H)$, then the kernel of $f$ is a closed subspace of $V$ and hence also an object in $\mathrm{Rep}^{\mathrm{ind-adm}}(H)$ by Corollary \ref{closed}. Likewise, we have already seen in Corollary \ref{iastrict} $\mathrm{coker}f=(\mathrm{ker}f')'$, and thus $\mathrm{coker}f$ is also ind-admissible. 
			
			Lastly, we have already seen in Corollary \ref{iastrict} that $f$ is strict, so that $\mathrm{coim}f\to \mathrm{im}f$ is an isomorphism. Thus $\mathrm{Rep}^{\mathrm{ind-adm}}(H)$ is an abelian category.
			\item[(ii)] The statement on closed subrepresentations is proved in Corollary \ref{closed}, the statement on countable direct sums is immediate from the definition.
		\end{enumerate}
	\end{proof}
	
	Of course, we have dually that the category of pro-coadmissible $D(H)$-modules is abelian, closed under the passage to closed submodules and under countable products.
	
	We also remark that all results for pro-coadmissible $D(C)$-modules hold in the same way when working over an arbitrary Fr\'echet--Stein algebra $U$, provided that coadmissible $U$-modules are nuclear Fr\'echet spaces (e.g. $U=\wideparen{U(\mathfrak{g})}$ for some finite-dimensional $K$-Lie algebra $\mathfrak{g})$, or at least nuclear relative to some Noetherian Banach $K$-algebra $A$ with (almost) Noetherian unit ball, in the sense of \cite[Definition 5.3]{SixOp} (e.g. $U=\wideparen{U_A(L)}$ for an affinoid $K$-algebra $A$ and a smooth $(K, A)$-Lie--Rinehart algebra $L$).
	
	We conclude this section by highlighting one additional benefit of a complete bornological approach in the study the distribution algebra.
	
	Recall from \cite{ST1} that if $H$ is non-compact, then the distribution algebra $D(H)$ is slightly awkward to work with: the multiplication is only separately continuous, and accordingly, we usually consider nuclear Fr\'echet $D(H)$-modules with a separately continuous action. 
	
	We point out that this issue disappears on the bornological side: Let
	\begin{equation*}
		D_b(H)=\underset{h\in H/C}{\oplus}hD(C)\in \B c_K.
	\end{equation*} 
	Note $D_b(H)$ is a complete bornological vector space by \cite[Proposition 5.5]{ProsmansSchneiders}, and the multiplication on $D(C)$ together with multiplication on $H$ induces on $D_b(H)$ the structure of a complete bornological $K$-algebra, i.e. a monoid in $\h{\B}c_K$.
	
	Moreover, recall that $(-)^b: LCVS_K\to \B c_K$ admits a left adjoint $(-)^t: \B c_K\to LCVS_K$. Since $(-)^t$ commutes with colimits, it follows immediately that
	\begin{equation*}
		(D_b(H))^t\cong D(H).
	\end{equation*}
	
	We remark that while the underlying vector space is the same, it is not true that $D_b(H)=D(H)^b$ if $H$ is non-compact, as $(-)^b$ does not respect direct sums. Rather $D_b(H)=\mathrm{Cpt}(D(H))$ is the vector space $D(H)$ endowed with its compactoid bornology, compare \cite[Definition 1.15]{Meyer}, \cite[Definition 3.34]{Stein}. We also point out that much of the nice behaviour of nuclear Fr\'echet spaces comes from the fact that for $V$ nuclear Fr\'echet, $V^b\cong \mathrm{Cpt}(V)$ (compare \cite[Lemma 5.11]{SixOp}, \cite[Proposition 19.2]{S1}).
	
	\begin{lemma}\label{embeddingfornF}
		The functor $(-)^b$ induces a fully faithful functor
		\begin{equation*}
			(-)^b: \left\{\begin{matrix}
				\text{separately continuous $D(H)$-modules}\\
				\text{on nuclear Fr\'echet spaces}\\
				\text{with continuous $D(H)$-module morphisms}
			\end{matrix}\right\}\to \mathrm{Mod}_{\h{\B}c_K}(D_b(H))
		\end{equation*}
		whose essential image consists of all complete bornological $D_b(H)$-modules whose underlying vector spaces are nuclear Fr\'echet.
		
		For any $M$ in the left-hand side, the natural morphism $(M^b)^t\to M$ is an isomorphism.
		If $M$ is nuclear Fr\'echet $D(H)$-module and $N$ is a closed submodule, then
		\begin{equation*}
			0\to N^b\to M^b\to (M/N)^b\to 0
		\end{equation*}
		is strictly exact.
	\end{lemma}
	
	\begin{proof}
		The above description makes it clear that a complete bornological $D_b(H)$-module is the same as a complete bornological $D(C)$-module such that the $C$-action extends to an abstract $H$-action with each $h\in H$ acting via a bounded linear map.
		
		The result thus follows from Lemma \ref{bembedding}.
	\end{proof}
	
	Viewing complete bornological $D_b(H)$-modules as complete bornological $D(C)$-modules with an extended $H$-action can be useful to define $D_b(H)$-module structures.
	
	For instance, if $C$ is normal in $H$ and $M$ is a complete bornological $D_b(H)$-module, then
	\begin{equation*}
		D_n(C)\h{\otimes}_{D(C)}M
	\end{equation*} 
	is likewise a complete bornological $D_b(H)$-module, with the $H$-action obtained by
	\begin{equation*}
		h*(P\otimes m)=hPh^{-1}\otimes hm.
	\end{equation*}
	
	We say that a nuclear Fr\'echet $D_b(H)$-module is topologically simple if it has no non-trivial closed submodules. By the above, this is equivalent to requiring that it has no non-trivial closed $H$-invariant $D(C)$-submodules.
	
	In particular, a pro-coadmissible $D_b(H)$-module $M$ is topologically simple if and only if the ind-admissible $H$-representation $M'$ is topologically irreducible. 
	
	The next Corollary can then be used to determine the topological irreducibility of ind-admissible representations, in analogy to \cite[Lemma 3.9]{ST2}.
	
	\begin{cor}
		\label{eachnirred}
		Assume that $C$ is normal in $H$. Let $M$ be a pro-coadmissible $D_b(H)$-module such that $D_n(C)\h{\otimes}_{D(C)}M$ is a topologically simple $D_b(H)$-module for infinitely many $n$. Then $M$ is a topologically simple $D_b(H)$-module.
	\end{cor}
	\begin{proof}
		Let $N$ be a closed $D_b(H)$-submodule of $M$. By Corollary \ref{closed}, it is pro-coadmissible, and Corollary \ref{Dnflat} implies that $D_n(C)\h{\otimes}_{D(C)} N$ is a closed $D_b(H)$-submodule of $D_n(C)\h{\otimes}_{D(C)} M$. 
		
		Thus by assumption, $D_n(C)\h{\otimes}_{D(C)} N$ is either zero or equal to the module $D_n(C)\h{\otimes}_{D(C)} M$ for infinitely many $n$. Since the maps
		\begin{equation*}
			D_{n+1}(C)\h{\otimes}_{D(C)}N\to D_n(C)\h{\otimes}_{D(C)} N
		\end{equation*}
		are epimorphisms (i.e. have dense image), it follows from Lemma \ref{changeton} that either $N=0$ (when $D_n(C)\h{\otimes}N=0$ for all $n$) or $N= M$. 
	\end{proof}
	
	A standard source of ind-admissible representation is compact induction. For this, let $G$ be a locally $L$-analytic group.

	\begin{dfn}
		\label{compactmodH}
		Let $H$ be a closed (not necessarily compact) subgroup of $G$, and let $V$ be a Hausdorff locally convex $K$-vector space. A locally analytic function $f: G\to V$ has \textbf{compact support modulo $H$} if there exists a compact subset $C\subseteq G$ such that $f(g)=0$ for all $g\notin CH$. 
	\end{dfn}
	
	Note that if $H$ itself is compact, having compact support modulo $H$ is equivalent to having compact support.
	
	For our purposes, the following situation will be more important: Suppose that $H_0$ is a compact open subgroup of $G$ and $H=H_0Z_G$, where $Z_G$ denotes the centre of $G$. In this case, having compact support modulo $H$ is equivalent to the more standard terminology of `compact support modulo centre', i.e. there exists a compact subset $C\subseteq G$ such that $f$ vanishes outside of $CZ_G$. By compactness of $H_0$, this is equivalent to the following: there exist finitely many elements $g_1, \hdots, g_m\in G$ such that $f(g)=0$ for all $g\notin \cup_{i=1}^m g_iH$.
	
	It is this viewpoint which will be most convenient for our constructions.
	
	Fix from now on an open compact subgroup $H_0$ and set $H=H_0Z_G$. We assume that the set $G/H$ of cosets is countable -- this is for example the case when $G$ is second countable.
	
	Let $\rho: H\to \mathrm{GL}(V)$ be a locally $L$-analytic $H$-representation over $K$. We consider the compact induction
	\begin{equation*}
		\mathrm{c-Ind}_H^GV=\left\{
		f: G\to V | \begin{matrix}\ f \text{ loc. an. with compact support modulo $H$},\\ \ f(gh)=\rho(h^{-1})(f(g)) \ \forall g\in G, \ h\in H
		\end{matrix}\right\}
	\end{equation*}
	with the usual action $(g*f)(g')=f(g^{-1}g')$. Note that $\mathrm{c-Ind}_H^GV$ is a locally $L$-analytic $G$-representation which as a topological $K$-vector space can be written as $\underset{g\in G/H}{\oplus} gV$, with the function $f$ corresponding to the (finite, well-defined) sum $\sum_{g\in G/H} gf(g)$.
	
	
	It is clear from the construction that this yields an additive functor
	\begin{equation*}
		\mathrm{c-Ind}_H^G: \mathrm{Rep}^{\mathrm{la}}(H)\to \mathrm{Rep}^{\mathrm{la}}(G).
	\end{equation*}
	
	\begin{lemma}\label{Mackeydecomp}
		\leavevmode
		\begin{enumerate}
			\item[(i)] $\mathrm{c-Ind}_H^G$ is a left adjoint to $\mathrm{Res}^G_H$.
			\item[(ii)] (Mackey decomposition). There is a natural isomorphism
			\begin{equation*}
				\mathrm{Res}^G_H \mathrm{c-Ind}_H^G V\cong \underset{g\in H\backslash G/H}{\oplus} \mathrm{c-Ind}_{H\cap gHg^{-1}}^H \mathrm{Res}_{H\cap gHg^{-1}}^{gHg^{-1}} gV.
			\end{equation*}
		\end{enumerate}
	\end{lemma}
	\begin{proof}
		\begin{enumerate}
			\item[(i)] This is Proposition \ref{Mackeyadj} in the case when $H$ is compact. The same argument works for $H$ compact mod centre.
			
			
			\item[(ii)] Note that
			\begin{equation*}
				\underset{g\in G/H}{\oplus} gV\cong \underset{g\in H\backslash G/H}{\oplus} \underset{h\in H/H\cap gHg^{-1}}{\oplus} hgV,
			\end{equation*}
			identifying the $H$-representation $\mathrm{Res}_H^G\mathrm{c-Ind}_H^GV$ with 
			\begin{equation*}
				\underset{g\in H\backslash G/H}{\oplus} \mathrm{c-Ind}_{H\cap gHg^{-1}}^H \mathrm{Res}_{H\cap gHg^{-1}}^{gHg^{-1}} gV. 
			\end{equation*}
		\end{enumerate}
	\end{proof}
	
	We remark that $gHg^{-1}=gH_0g^{-1}Z_G$, so $H\cap gHg^{-1}$ is actually of finite index in $H$. We could thus replace the $\mathrm{c-Ind}_{H\cap gHg^{-1}}^H$ above by $\mathrm{Ind}_{H\cap gHg^{-1}}^H$.
	
	\begin{lemma}
		\label{india}
		If $V$ is an admissible $H$-representation, then $\mathrm{c-Ind}_H^GV$ is an ind-admissible $G$-representation.
	\end{lemma}
	\begin{proof}
		By the Mackey decomposition above, it suffices to show that 
		\begin{equation*}
			\mathrm{c-Ind}_{H\cap gHg^{-1}}^H \mathrm{Res}_{H\cap gHg^{-1}}^{gHg^{-1}} gV=\mathrm{Ind}_{H\cap gHg^{-1}}^H \mathrm{Res}_{H\cap gHg^{-1}}^{gHg^{-1}} gV
		\end{equation*}
		is an admissible $H$-representation for each $g$.
		
		Clearly, $gV$ is an admissible $gHg^{-1}$-representation. In particular, $(gV)'$ is a coadmissible module over the distribution algebra $D(H_0\cap gH_0g^{-1})$, as $H_0\cap gH_0g^{-1}$ is a compact open subgroup. Now
		\begin{equation*}
			(\mathrm{Ind}_{H\cap gHg^{-1}}^H \mathrm{Res}_{H\cap gHg^{-1}}^{gHg^{-1}}gV)'\cong D(H)\h{\otimes}_{D(H\cap gHg^{-1})}(gV)'
		\end{equation*}
		is a Hausdorff quotient of
		\begin{equation*}
			D(H_0)\h{\otimes}_{D(H_0\cap gH_0g^{-1})}(gV)',
		\end{equation*}
		since $H_0/H_0\cap gH_0g^{-1}$ surjects onto $H/H\cap gHg^{-1}$ by the remark above. Since $D(H_0)$ is a finite free $D(H_0\cap gH_0g^{-1})$-module,
		\begin{equation*}
			D(H_0)\h{\otimes}_{D(H_0\cap gH_0g^{-1})} gV'=D(H_0)\otimes_{D(H_0\cap gH_0g^{-1})} gV'
		\end{equation*}  
		is a coadmissible $D(H_0\cap gH_0g^{-1})$-module and hence also coadmissible as a $D(H_0)$-module, which proves the result.
	\end{proof}
	
	We remark that the dual $(\mathrm{c-Ind}_{H}^GV)'_b$ is then naturally isomorphic to the pro-coadmissible $D_b(H)$-module 
	\begin{align*}
		\mathrm{Hom}_{D_b(H)}(D_b(G), V')&\cong \prod_{g\in G/H} gV'\\
		&\cong \prod_{g\in H\setminus G/H} D(H)\otimes_{D(H\cap gHg^{-1})} gV'
	\end{align*}
	due to the decomposition above.

	\subsection{Hecke operators in the case $G=\mathrm{GL}_2(L)$}
	
	In this appendix, we describe some non-trivial Hecke operators on the compactly induced $G=\mathrm{GL}_2(L)$-representation
	\begin{equation*}
		\mathrm{c-Ind}_{G_0Z_G}^G ((D(G_0)\otimes_{D(\mathfrak{g}, T_0)}M^\mu)')
	\end{equation*}
	from the main body of the text. It follows in particular that this representation is not topologically irreducible, but that one should rather study the constituents obtained by fixing a Hecke character.
	
	While we remark the obvious similarity to \cite{BarLiv} in the modular case, determining the precise shape of the Hecke algebra in our case seems more subtle -- while we show that the Hecke algebra contains the polynomial ring $K[T]$, we cannot rule out the existence of further Hecke operators.
	
	We quickly recall the notations and definitions from the main body of the text. 
	
	Let $L/\mathbb{Q}_p$ be a finite field extension and $K/L$ a spherically complete field. Let $G=\mathbf{GL}_2(L)$, and let $T=\mathbf{T}(L)$ denote the maximal torus of diagonal elements. Let $G_0=\mathbf{GL}_2(O_L)$, $T_0=\mathbf{T}(O_L)$. Lie algebras over $K$ will be denoted by $\mathfrak{g}$ resp. $\mathfrak{t}$.

	Fix a locally analytic character $\mu: T_0\to K^\times$, which differentiates to a character $\mathrm{d}\mu: \mathfrak{t}\to K$.
	
	Upon picking a basis
	\begin{equation*}
		h_0=\begin{pmatrix}
			1&0\\0&0
		\end{pmatrix}, \ h_1=\begin{pmatrix}
			0&0\\0&1
		\end{pmatrix}
	\end{equation*} 
	for $\mathfrak{t}$, we can record $\mathrm{d}\mu$ as a pair $\mathrm{d}\mu=(\mu_0, \mu_1)$ with $\mu_i=\mathrm{d}\mu(h_i)$. (This is the only change of notation compared to the main body of the text, where $\mathrm{d}\mu$ is called $\mu$ and our $\mu$ is called $\chi_\mu$.)
	
	If $x_0, x_1$ are polynomial variables, we write $x^\mu$ for the formal element
	\begin{equation*}
		x^\mu=x_0^{\mu_0}x_1^{\mu_1}
	\end{equation*}
	and consider the following $U(\mathfrak{g})$-module:
	
	\begin{equation*}
		M^\mu:=\{f\cdot x^\mu\in K[x_0^{\pm1}, x_1^{\pm 1}]\cdot x^\mu| \ f \text{ homogeneous of degree $0$}\},
	\end{equation*}
	where $e$ acts as $x_0\partial_1$, $f$ acts as $x_1\partial_0$, $h_0$ acts as $x_0\partial_0$ and $h_1$ acts as $x_1\partial_1$.
	
	We can further define a diagonal $T_0$-action on $M^\mu$ given by 
	\begin{equation*}
		\begin{pmatrix}
			a&0\\0&d
		\end{pmatrix}\cdot x_0^ix_1^{-i}\cdot x^{\mu}=\mu(\begin{pmatrix}
			a&0\\0&d
		\end{pmatrix})a^id^{-i} x_0^ix_1^{-i}x^\mu.
	\end{equation*}
	
	In other words, if we let $\rho: \O_L^\times \to K^\times$ denote the natural embedding character and $\rho_j: T_0\to K^\times$ the character obtained by composing the projection to the $(j, j)$-entry with $\rho$, then $T_0$ acts on $Kx_0^ix_1^{-i}x^\mu$ via the locally analytic character $\rho_0^i\rho_1^{-i}\mu$.
	
	As a $T_0$-representation, $M^\mu$ is simply a direct sum of locally analytic characters, and in particular we can regard it as an abstract $D(T_0)$-module.
	
	Since $\rho_0^i\rho_1^{-i}\mu$ differentiates to the weight $(\mu_0+i, \mu_1-i)$, the $T_0$-action differentiates to the action of $\mathfrak{t}\subseteq \mathfrak{g}$. Furthermore, it is straightforward to verify that the $\mathfrak{g}$-action is $T_0$-equivariant, endowing $M^\mu$ with a $D(\mathfrak{g}, T_0)$-module structure.
	
	Let $\mathfrak{z}=Z(U(\mathfrak{g}))$ denote the centre of the universal enveloping algebra of $\mathfrak{g}$ and let $\chi: \mathfrak{z}\to K$ be the central character corresponding to the weight $h_0-h_1\mapsto \mu_0+\mu_1$, $h_0+h_1\mapsto \mu_0+\mu_1$.
	Explicitly, $\chi$ is obtained via the Harish-Chandra isomorphism as
	\begin{equation*}
		\mathfrak{z}\cong U(\mathfrak{t})^W\to U(\mathfrak{t})\to K,
	\end{equation*}
	sending the Casimir element $(h_0-h_1)^2+2(h_0-h_1)+4fe$ to $(\mu_0+\mu_1)\cdot (\mu_0+\mu_1+2)$ and $h_0+h_1$ to $\mu_0+\mu_1$.
	
	Since
	\begin{equation*}
		(h_0+h_1)\cdot x^\mu=(\mu_0+\mu_1)x^\mu
	\end{equation*}
	and
	\begin{align*}
		((h_0-h_1)^2+(h_0-h_1)+4fe)x^{\mu}&=((\mu_0-\mu_1)^2+2(\mu_0-\mu_1)+4(\mu_0+1)\mu_1)x^{\mu}\\
		&=(\mu_0+\mu_1)\cdot (\mu_0+\mu_1+2)\cdot x^{\mu},
	\end{align*}
	we know that $\mathfrak{z}$ acts on $M^{\mu}$ via $\chi$.
	
	We now give an explicit presentation for $M^\mu$ as a $D(\mathfrak{g}, T_0)$-module.
	
	\begin{lemma}
		\label{presentation}
		Suppose that $\mu_0, \mu_1\notin \mathbb{Z}$.
		
		There is a natural isomorphism
		\begin{equation*}
			M^{\mu}\cong D(\mathfrak{g}, T_0)/I,
		\end{equation*}
		where $I\subseteq D(\mathfrak{g}, T_0)$ is the left ideal generated by $t-\mu(t)$ for $t\in T_0$ and 
		\begin{equation*}
			z-\chi(z), \ z\in \mathfrak{z}.
		\end{equation*} 
	\end{lemma}
	\begin{proof}
		Let $J\subseteq U(\mathfrak{g})$ be the left ideal generated by $h_0-\mu_0$, $h_1-\mu_1$ and by $z-\chi(z)$ for $z\in \mathfrak{z}$.
		
		It follows from the above that the map $U(\mathfrak{g})\to M^{\mu}$ sending $1$ to $x^{\mu}$ factors through $U(\mathfrak{g})/J$ and is surjective (surjectivity follows from $\mu_0, \mu_1\notin \mathbb{Z}$).
		
		This induces a good filtration on $M^{\mu}$, with $x_0^jx_1^{-j}x^{\mu}$ of degree $|j|$. The associated graded map
		\begin{equation*}
			\mathrm{gr}(U(\mathfrak{g})/J)\cong \frac{K[e, f, h_0, h_1]}{(((h_0-h_1)^2+4ef), h_0, h_1)}\cong \frac{K[e, f]}{ef}\to \mathrm{gr}M^{\mu}
		\end{equation*}
		is easily seen to be a bijection, and hence
		\begin{equation*}
			U(\mathfrak{g})/J\cong M^{\mu}
		\end{equation*}
		as $U(\mathfrak{g})$-modules via the map above.
		
		Since $x^{\mu}$ is a $T_0$-eigenvector with weight $\mu$, the result follows.
	\end{proof}  
	
	From now on, we always assume that $\mu_0, \mu_1\notin \mathbb{Z}$. We now consider $M_0=D(G_0)\otimes_{D(\mathfrak{g}, T_0)}M^\mu$ -- this is a coadmissible (finitely presented) $D(G_0)$-module, with a presentation analogous to the one above.

	Choose a locally analytic character $Z_G\to K^\times$ extending $\mu$ on $Z_G\cap T_0$ to turn $V=(M_0)'_b$ into an admissible, locally analytic $G_0Z_G$-representation.
	
	We set $W=\mathrm{c-Ind}_{G_0Z_G}^GV$ and consider the Hecke algebra $\mathcal{H}=\mathrm{End}_G(W)^{\mathrm{op}}$. For the rest of the appendix, all Hom spaces over $K$ will be considered as internal homs in the category of complete bornological spaces. As we will only be concerned with nuclear Fr\'echet $D_b(G)$-modules, this is the same as considering continuous morphisms by Lemma \ref{embeddingfornF} and \cite[Lemma 4.3]{SixOp}, but it has the advantage of also endowing our Hom spaces with natural bornological structures.
	
	To understand $\mathcal{H}$, note that $W$ is ind-admissible (Proposition \ref{ind-stradm}), so if we write $M=W'_b$, then
	\begin{equation*}
		\mathcal{H}\cong \mathrm{End}_{D_b(G)}(M)\cong \mathrm{Hom}_{D(G_0Z_G)}(M, M_0),
	\end{equation*}
	and the Mackey decomposition from Lemma \ref{Mackeydecomp} yields
	\begin{equation*}
		\mathcal{H}\cong \mathrm{Hom}_{D(G_0Z_G)} (\prod_{s\in G_0\setminus G/G_0Z_G} D(G_0)\otimes_{D(G_0\cap sG_0s^{-1})}sM_0, M_0).
	\end{equation*}
	As the action of the centre is fixed throughout, we can replace $D(G_0Z_G)$ with $D(G_0)$.
	
	We remark that $\mathrm{End}_G(W)$ can be identified with the algebra of compactly supported, $G_0Z_G$-biequivariant functions
	\begin{equation*}
		\mathrm{Hom}_c^{G_0-G_0Z_G}(G, \mathrm{Hom}_K(V, V))
	\end{equation*}
	as follows: if $\theta: W\to W$ is a (continuous, i.e. bounded) endomorphism, note that for each $v\in V\subseteq W$, $\theta(v)$ can be regarded as a compactly supported function $G\to V$. Then $\varphi_\theta(g)$ is defined by $v\mapsto \theta(v)(g)$. Conversely, if $\varphi: G\to \mathrm{Hom}_K(V, V)$, then $v\mapsto (g\mapsto \varphi(g)(v))$ determines a (continuous, i.e. bounded) $G_0Z_G$-linear map $V\to W=\mathrm{c-Ind}_{G_0Z_G}^GV$, and hence a $G$-linear map $W\to W$ by adjunction. The multiplication is then given by convolution, i.e.
	\begin{equation*}
		(\varphi_1*\varphi_2)(g)=\sum_{h\in G/G_0Z_G} \varphi_1(h)\circ \varphi_2(h^{-1}g).
	\end{equation*}
	
	Since $V$ is reflexive with $V'=M_0$, it follows that $\mathrm{End}_K(V)=\mathrm{End}_K(M_0)^{\mathrm{op}}$. For $f\in \mathrm{End}_K(V)$, we let $f^*\in \mathrm{End}_K(M_0)$ denote the dual map. For
	\begin{equation*}
		\varphi\in \mathrm{Hom}_c^{G_0-G_0Z_G}(G, \mathrm{End}_K(V)),
	\end{equation*}
	write $\varphi^*: G\to \mathrm{End}_K(M_0)$ for the map $\varphi^*(g)=\varphi(g^{-1})^*$. This exhibits the convolution algebra $\mathrm{Hom}_c^{G_0-G_0Z_G}(G, \mathrm{End}_K(M_0))$ as the opposite algebra to $\mathrm{Hom}_c^{G_0-G_0Z_G}(G, \mathrm{End}_K(V))$, so that $\mathcal{H}=\mathrm{End}_G(W)^{\mathrm{op}}=\mathrm{End}_{D_b(G)}(M)$ can be identified accordingly with 
	\begin{equation*}
		\mathrm{Hom}_c^{G_0-G_0Z_G}(G, \mathrm{Hom}_K(M_0, M_0))
	\end{equation*}
	under convolution. The involution $g\mapsto g^{-1}$ which is involved in this construction is the same as in \cite{ST2} to turn $M=W'_b$ into a left $D(G)$-module.
	
	A set of double coset representatives for the Mackey decomposition is given by 
	\begin{equation*}
		t_i=\begin{pmatrix}
			1&0\\
			0& \pi^i
		\end{pmatrix}, \ i\geq 0.
	\end{equation*} 
	
	\begin{prop}
		Assume that $\mu_0, \mu_1\notin \mathbb{Z}$.
		
		For each $i\geq 0$, there exists a non-zero Hecke operator $T_i\in \mathcal{H}$ which is attached to $t_i$, in the sense that the support of the associated function
		\begin{equation*}
			G\to \mathrm{End}_K(M_0)
		\end{equation*}
		is contained in the double coset $G_0t_iG_0Z_G$. 
	\end{prop}
	\begin{proof}
		Consider the subalgebra $\mathfrak{z}\cdot D(T0)\subset D(G_0)$, and let $K_{\chi, \mu}$ denote the $1$-dimensional $\mathfrak{z}\cdot D(T_0)$-module given by the $T_0$-character $\mu$ and the central character $\chi$. Note that Lemma \ref{presentation} yields isomorphisms
		\begin{equation*}
			M^\mu\cong D(\mathfrak{g}, T_0)\otimes_{\mathfrak{z} D(T_0)}K_{\chi, \mu}
		\end{equation*}
		and
		\begin{equation*}
			M_0\cong D(G_0)\otimes_{\mathfrak{z}D(T_0)} K_{\chi, \mu}.
		\end{equation*}
		
		Accordingly, we identify $K_{\chi, \mu}$ with the one-dimensional $\mathfrak{z}D(T_0)$-submodule $K\cdot x^\mu\subseteq M_0\subseteq M$. 
		
		Fix $i\geq 0$ and consider the one-dimensional subspace $t_iK x^\mu=t_i K_{\chi, \mu}\subseteq M$. As $t_i$ commutes with the centre $\mathfrak{z}$ of the universal enveloping algebra and with elements of $T_0$, $t_iK_{\chi, \mu}$ is again a $\mathfrak{z}D(T_0)$-submodule of $M$ and there is a tautological isomorphism
		\begin{equation*}
			t_iK_{\chi, \mu}\cong K_{\chi, \mu},
		\end{equation*}
		sending $t_i x^\mu$ to $x^\mu$. This yields a non-trivial map
		\begin{equation*}
			D(G_0\otimes t_iG_0t_i^{-1})\otimes_{t_i\mathfrak{z}D(T_0)t_i^{-1}} t_iK_{\chi, \mu}=D(G_0\cap t_ig_0t_i^{-1})\otimes_{\mathfrak{z}D(T_0)}t_iK_{\chi, \mu}\to D(G_0)\otimes_{\mathfrak{z}D(T_0)} K_{\chi, \mu}.
		\end{equation*}
		
		Now as a $D(G_0\cap t_i^{-1}G_0t_i)$-module, we can write
		\begin{equation*}
			M_0\cong M'_i\oplus (\underset{s\in G_0/G_0\cap t_i^{-1}G_0t_i, \ s\neq e}{\bigoplus} sM'_i)
		\end{equation*}
		for 
		\begin{equation*}
			M'_i=D(G_0\cap t_i^{-1}G_0t_i)\otimes_{\mathfrak{z}D(T_0)}K_{\chi, \mu},
		\end{equation*}
		and as a $D(G_0\cap t_iG_0t_i^{-1})$-module, we can write
		\begin{equation*}
			M_0\cong M_i\oplus (\underset{s\in G_0/G_0\cap t_iG_0t_i^{-1}, \ s\neq e}{\bigoplus} sM_i)
		\end{equation*}
		for
		\begin{equation*}
			M_i=D(G_0\cap t_iG_0t_i^{-1})\otimes_{\mathfrak{z}D(T_0)} K_{\chi, \mu}.
		\end{equation*}
		
		In particular, the natural submodule
		\begin{equation*}
			t_iM'_i\cong D(G_0\cap t_iG_0t_i^{-1})\otimes t_iK_{\chi, \mu}\to D(t_iG_0t_i^{-1})\otimes t_iK_{\chi, \mu}=t_iM_0
		\end{equation*}
		is a direct summand as a $D(G_0\cap t_iG_0t_i^{-1})$-module, so that we obtain a non-zero $D(G_0\cap t_iG_0t_i^{-1})$-linear map
		\begin{equation*}
			\alpha_i: t_iM_0\to D(G_0\cap t_iG_0t_i^{-1})\otimes t_iK_{\chi, \mu}\to M_0
		\end{equation*}
		by composing the above with a corresponding projection map.
		
		By adjunction, we obtain a non-zero $D(G_0)$-linear map
		\begin{equation*}
			T'_i: D(G_0)\otimes_{D(G_0\cap t_iG_0t_i^{-1})} t_iM_0\to M_0,
		\end{equation*}
		which we can regard as an element of $\mathcal{H}$ using the Mackey decomposition and adjunction: writing $M\cong \prod_{g\in G/G_0Z_G} gM_0$, we have $T_i(m)=(gT'_i(g^{-1}m))_g$.
		
		Explicitly, we can also view $T_i$ as a $G_0Z_G$-biequivariant function
		\begin{equation*}
			G\to \mathrm{Hom}_K(M_0)
		\end{equation*}
		as follows: the function is supported on the double coset $G_0t_iG_0Z_G$, sending $t_i$ to the endomorphism
		\begin{align*}
			\varphi_i: &M_0\to M_0\\
			&m\mapsto \alpha_i(t_im).
		\end{align*}
		
		Note that this does indeed give a well-defined function on $G_0t_iG_0Z_G$, since
		\begin{equation*}
			\varphi_i (gm)=\alpha_i(t_igm)=t_igt_i^{-1}\alpha_i(t_im)=(t_igt_i^{-1})\varphi_i(m)
		\end{equation*}
		for $g\in t_i^{-1}G_0t_i\cap G_0$, thanks to the $D(G_0\cap t_iG_0t_i^{-1})$-linearity of $\alpha_i$.
		
		In particular, it is immediate that the support of $T_i$ is contained in the double coset of $t_i$.
	\end{proof}
	
	As each $T_i$ is supported on a different double coset, it follows that these operators are linearly independent.
	
	We remark that now that $T_i$ is well-defined, there is a straightforward way to describe it: as $M$ is topologically generated by $x^\mu$, the $D_b(G)$-linear map $T_i: M\to M$ is uniquely determined by $T_i(x^\mu)$. Writing $M\cong \prod_{g\in G/G_0Z_G} gM_0$, we have as before
	\begin{equation*}
		T_i(x^\mu)=(g T'_i(g^{-1}x^\mu))_g.
	\end{equation*}
	Note that $T'_i$ projects to $D(G_0)\otimes_{D(G_0\cap t_iG_0t_i^{-1})}t_iM_0$ (so that the only $g\in G/G_0Z_G$ yielding non-zero entries for $T_i(x^\mu)$ satisfy $g^{-1}\in G_0t_iG_0Z_G$), and from there further to
	\begin{equation*}
		D(G_0)\otimes_{D(G_0\cap t_iG_0t_i^{-1})}(D(G_0\cap t_iG_0t_i^{-1})\otimes_{\mathfrak{z}D(T_0)}t_iK_{\chi, \mu})=D(G_0)\otimes_{\mathfrak{z}D(T_0)}t_iK x^\mu,
	\end{equation*}
	so that $T'_i(g^{-1}x^\mu)=0$ for any $g\in G/G_0Z_G$ with $g\notin t_i^{-1}G_0Z_G$. Hence,
	\begin{equation*}
		T_i(x^\mu)=t_i^{-1}T'_i(t_i x^\mu)=t_i^{-1}x^\mu.
	\end{equation*}
	
	As one consequence, we can deduce that $M$, and hence $W$, is not topologically irreducible: Let $\lambda\in K^\times$. The endomorphism $T_1-\lambda\in \mathrm{End}_{D_b(G)}(M)$ is non-zero, and its kernel contains the element
	\begin{equation*}
		(\lambda^n t_1^n x^\mu)_{n\in \mathbb{Z}}\in \prod_{g\in G/G_0Z_G} gM_0=M.
	\end{equation*}
	We therefore obtain a non-trivial closed $D_b(G)$-submodule.
	
	The characterisation above also makes the following observation immediate.
	
	\begin{cor}
		We use the same notation as above. The $T_i$ satisfy $T_i\cdot T_j=T_{i+j}$ for all $i, j\geq 0$. Setting $T=T_1$, this exhibits the polynomial ring $K[T]$ as a subring of $\mathcal{H}$.
	\end{cor}

	It is natural to ask whether, as in \cite{BarLiv}, the Hecke algebra $\mathcal{H}$ is in fact isomorphic to the polynomial ring in those cases when $M_0$ is topologically simple. This requires a proof that the subspace of functions which are supported on any given double coset $G_0t_iG_0Z_G$ is one-dimensional. This seems not clear at all. 
	
	It might seem strange at first to obtain a polynomial ring in the context of complete bornological algebras (and not a ring of convergent power series, say), so we should briefly explain why this is not entirely surprising: Let $D(G_0)=\varprojlim D_n(G_0)$, suppose that $N\cong N_n$ is a coadmissible $D(G_0)$-module and $M=\prod_i M^{(i)}$ is a countable product of coadmissible $D(G_0)$-modules, $M^{(i)}\cong \varprojlim M^{(i)}_n$. In our setting, we have $N=M_0$, and $M=\prod D(G_0)\otimes_{D(G_0\cap t_iG_0t_i^{-1})} t_i M_0$. We then have
	\begin{align*}
		\mathrm{Hom}_{D(G_0)}(M, N)&\cong \varprojlim \mathrm{Hom}_{D(G_0)}(\prod M^{(i)}, N_n)\\
		&\cong \varprojlim_n \mathrm{Hom}_{D_n(G_0)}(\prod M^{(i)}_n, N_n)\\
		&\cong \varprojlim_n \oplus_i \mathrm{Hom}_{D_n(G_0)}(M^{(i)}_n, N_n)
	\end{align*}
	by Lemma \ref{changeton} and the fact that any continuous morphism from $\prod M^{(i)}_n\cong \varprojlim_m \prod_{i=1}^m M^{(i)}_n$ to some Banach space needs to factor through one of the $\prod_{i=1}^m M^{(i)}_n$. In particular, if this inverse system is constant, we obtain a complete bornological $K$-vector space which is endowed with a direct sum bornology. E.g. if $\mathrm{Hom}_{D_n(G_0)}(M^{(i)}_n, N_n)$ is one-dimensional for all $i$ and $n$ (with isomorphisms as connecting morphisms), we obtain the vector space $K[T]$ endowed with the bornology
	\begin{equation*}
		K[T]\cong \oplus_i KT^i,
	\end{equation*}
	meaning a subset is bounded if and only if it is contained in a finite-dimensional $K$-vector subspace and is bounded with respect to the standard topology for finite-dimensional vector spaces.
	
	It would be very interesting to determine the precise shape of $\mathcal{H}$ and describe the corresponding quotients $W\otimes_{\mathcal{H}, \rho}K$ for a fixed Hecke character $\rho: \mathcal{H}\to K$. Unlike $W$ itself, these representations would also have a much higher chance of being admissible.

		\vskip18pt

		\bibliographystyle{plain}
		\bibliography{Cuspidal}
		
	\end{document}